\documentclass[11pt]{amsart}
\usepackage{amsmath}
\usepackage{amsfonts}
\usepackage{amssymb}
\usepackage{mathrsfs}                  
\usepackage{amsthm}
\usepackage{enumitem}
\usepackage{amscd}
\usepackage{xcolor}
\usepackage[hyperfootnotes=false]{hyperref}
\usepackage{mathtools}
\usepackage{geometry}
\usepackage{bbm}
\usepackage{bm}
\usepackage{bbm}

\numberwithin{equation}{section}
\theoremstyle{plain}
\newtheorem{theorem}{Theorem}[section]
\newtheorem{proposition}[theorem]{Proposition}
\newtheorem{lemma}[theorem]{Lemma}

\newtheorem{assumption}{Assumption}

\theoremstyle{definition}
\newtheorem{definition}[theorem]{Definition}

\theoremstyle{remark}
\newtheorem{remark}[theorem]{Remark}

\newcommand{\ReLU}{{\rm ReLU}}
\newcommand{\eps}{{\varepsilon}}

\newcommand{\rev}[1]{\textcolor{black}{#1}}

\newcommand{\cA}{{\mathcal A}}
\newcommand{\cB}{{\mathcal B}}
\newcommand{\cE}{{\mathcal E}}
\newcommand{\cN}{{\mathcal N}}
\newcommand{\cO}{{\mathcal O}}
\newcommand{\IR}{{\mathbb R}}
\newcommand{\IC}{{\mathbb C}}
\newcommand{\IE}{{\mathbb E}}
\newcommand{\IN}{{\mathbb N}}
\newcommand{\IP}{{\mathbb P}}

\newcommand{\abs}[1]{\left|#1\right|}
\newcommand{\bsnu}{{\boldsymbol \nu}}

\newcommand{\bse}{{\boldsymbol e}}
\newcommand{\bsb}{{\boldsymbol b}}

\newcommand{\E}{\mathbb{E}}
\newcommand{\R}{\mathbb{R}}
\newcommand{\Rl}{\mathrm{R}}
\newcommand{\C}{\mathbb{C}}

\newcommand{\N}{\mathbb{N}}

\renewcommand{\P}{\mathbb{P}}

\newcommand{\Fc}{\mathcal{F}}

\newcommand{\Gc}{\mathcal{G}}

\newcommand{\Hc}{\mathcal{H}}

\newcommand{\normc}[2][]{\left\| #2 \right\|_{#1}}

\begin{document}

\title[DNN expression rates for high-dimensional 
       exp-L\'evy models]{Deep ReLU Network Expression Rates for Option Prices
       \\
       in high-dimensional, exponential L\'evy models}
\author{Lukas Gonon}
\address{Department of Mathematics, University of Munich, Theresienstrasse 39, 80333 Munich, Germany}
\email{gonon@math.lmu.de}

\author{Christoph Schwab}
\address{Seminar for Applied Mathematics, ETH Z\"urich, R\"amistrasse 101, CH-8092 Z\"urich, Switzerland}
\email{christoph.schwab@sam.math.ethz.ch}
%\subjclass[2020]{ 68T07,   %Neural Networks
%                               60G51%Levy Processes
%}
\maketitle
\begin{abstract}
We study the expression rates of deep neural networks
(DNNs for short) 
for option prices written on baskets of $d$ risky assets, 
whose log-returns are modelled by a multivariate L\'evy process
with general correlation structure of jumps. 
We establish sufficient conditions on the characteristic triplet
of the L\'evy process $X$ that ensure $\varepsilon$ error of
DNN expressed option prices 
with DNNs of size that grows polynomially with respect to $\cO(\varepsilon^{-1})$, 
and with constants implied in $\cO(\cdot)$
which grow polynomially in $d$, 
thereby overcoming the curse of dimensionality (CoD)
and justifying the use of DNNs in financial modelling 
of large baskets in markets with jumps.

In addition, we exploit parabolic smoothing of Kolmogorov partial
integrodifferential equations for certain multivariate L\'evy processes
to present alternative architectures of ReLU DNNs 
that provide $\varepsilon$ expression error 
in DNN size $\cO(|\log(\varepsilon)|^a)$
with exponent $a \sim d$, however, with constants implied 
in $\cO(\cdot)$ growing exponentially with respect to $d$.
Under stronger, dimension-uniform  non-degeneracy conditions
on the L\'evy symbol, 
we obtain algebraic expression rates of option prices 
in exponential L\'evy models which are 
free from the curse of dimensionality.
In this case the ReLU DNN expression rates of prices 
depend on certain sparsity conditions on the characteristic L\'evy triplet.
We indicate several consequences and possible extensions of the 
present results.
\end{abstract}
\tableofcontents
%%%%%%%%%%%%%%%%%%%%%%%%%%%%%%%%%%%%%%%%%%%%%%%%%%%%%%%%
\section{Introduction}
\label{sec:Intro}

Recent years have seen a dynamic development in applications
of deep neural networks (DNNs for short) 
in expressing high-dimensional input-output
relations. This development was driven mainly by the need 
for quantitative modelling of input-output relationships 
subject to large sets of observation data.
Rather naturally, therefore, DNNs have found a large number
of applications in computational finance and in financial engineering.
We refer to the survey Ruf and Wang \cite{Ruf2020} and to the references there.
Without going into details, we only state that the majority of 
activity addresses techniques to employ DNNs in demanding 
tasks in computational finance.
The often striking efficient computational performance of DNN based
algorithms raises naturally the question for theoretical, in particular
mathematical, underpinning of successful algorithms. 
Recent years have seen progress, in particular in the context of
option pricing for Black--Scholes type models, 
for DNN based numerical approximation of
diffusion models on possibly large baskets 
(see, e.g.\ Berner et al.\ \cite{BernerGrohsJentzen2018}, Elbr\"achter et al.\ \cite{EGJS18_787}
and Ito et al.\ \cite{ito2019neural}, Reisinger and Zhang \cite{ReisingerZhang2019} for game-type options).
These references prove that DNN based approximations of
option prices on possibly large baskets of risky assets 
can overcome the so-called curse of dimensionality in the
context of affine diffusion models for the dynamics of the (log-)prices of
the underlying risky assets.
These results could be viewed also as particular instances of
DNN expression rates of certain PDEs on high-dimensional state spaces,
and indeed corresponding DNN expressive power results have been shown
for their solution sets in Grohs et al.\ \cite{HornungJentzen2018}, Gonon et al.\ \cite{Gonon2019}
and the references there.

Since the turn of the century, models beyond the classical
diffusion setting have been employed increasingly in 
financial engineering. 
In particular, L\'{e}vy processes
and their non-stationary generalizations such as Feller--L\'evy processes
(see, e.g., B\"ottcher et al.\ \cite{RSchillingEtAl2013} and the references there)
have received wide attention. 
This can in part be explained by their ability to account for heavy tails of financial data 
and by L\'evy--based models
constituting \emph{hierarchies} of models, comprising in particular
classical diffusion (``Black--Scholes'') models with constant 
volatility that are still widely used in computational finance as a benchmark.
Therefore, 
all results for geometric L\'evy processes in the present paper 
apply in particular to the Black--Scholes model.

The ``Feynman--Kac correspondence'' 
which relates conditional expectations of sufficiently regular functionals
over diffusions to (viscosity) solutions of corresponding Kolmogorov PDEs,
extends to multivariate L\'{e}vy processes.
We mention only 
Nualart and Schoutens \cite{NualSchout2001}, Cont and Tankov \cite{Cont2004}, Cont and Voltchkova \cite{ContVolt2005}, Glau \cite{FnmKacLevyGlau}, Eberlein and Kallsen
\cite[Chapter~5.4]{EberKall19} and the references there.
The Kolmogorov PDE (``Black--Scholes equation'')
in the diffusion case is then replaced by a so-called 
\emph{Partial Integrodifferential Equation (PIDE)}  where the 
fractional integrodifferential operator accounting for the jumps
is related in a one-to-one fashion with the L\'evy measure 
$\nu^d$ of the $\R^d$-valued LP $X^d$.
In particular, L\'evy type models for (log-)returns of 
risky assets result in 
\emph{nonlocal partial integrodifferential equations} for the option price,
which generalize the linear parabolic differential equations which arise
in classical diffusion models.
We refer to Bertoin \cite{BertoinLP1996}, Sato \cite{Sato1999}
for fundamentals on L\'{e}vy processes and 
to B\"ottcher et al.\ \cite{RSchillingEtAl2013} for extensions 
to certain non-stationary settings.
For the use of 
L\'evy processes in financial modelling we refer to 
Cont and Tankov \cite{Cont2004}, Eberlein and Kallsen \cite{EberKall19} and to the references there.
We refer to 
Cont and Voltchkova \cite{ContVolt2005,ContVolt2006},  Matache et al.\ \cite{MvS04_373}, Hilber et al.\ \cite{CMQFBook} for a presentation and for numerical methods for option pricing
in L\'evy models.

The results on DNNs in the context 
of option pricing mentioned above are 
exclusively concerned with models with continuous price processes. 
This naturally raises the question whether DNN based approximations are still 
capable of overcoming the curse of dimensionality in high-dimensional financial 
models with jumps, which have a much richer mathematical structure. 
This question is precisely the subject of this article.
We study the expression rates of DNNs
for prices of options (and the associated PIDEs) written on possibly large baskets of risky assets, 
whose log-returns are modelled by a multivariate L\'evy process
with general correlation structure of jumps. In particular, 
we establish sufficient conditions on the characteristic triplet
of the L\'evy process $X^d$ that ensure $\varepsilon$ error of
DNN expressed option prices 
with DNNs of size $\cO(\varepsilon^{-2})$, 
and with constants implied in $\cO(\cdot)$
which grow polynomially with respect $d$. 
This shows that DNNs are capable to overcome the curse of dimensionality 
also for general exponential L\'evy models. 

Let us outline the scope of our results.
The DNN expression rate results proved here 
give a theoretical justification for neural network 
based non-parametric option pricing methods. 
These have become very popular recently, see for instance the recent survey Ruf and Wang \cite{Ruf2020}.
Our results show that if option prices result from an exponential L\'evy model, 
as described e.g.\ in \cite[Chapter~3.7]{EberKall19},
under mild conditions on the L\'evy-triplets 
these prices can be expressed efficiently by (ReLU) neural networks,
also for high dimensions. 
The result covers, in particular, rather general, multivariate correlation structure in the 
jump part of the L\'{e}vy process, 
for example parametrized by a so-called 
\emph{L\'{e}vy copula}, see Kallsen and Tankov \cite{KT06}, Farkas et al.\ \cite{FRS07}, \cite[Chapter~8.1]{EberKall19}
and the references there.
This extends, at least to some extent, the theoretical foundation 
to the widely used neural network based non-parametric option pricing methodologies
to market models with jumps.

We prove two types of results on DNN expression rate bounds for
European options in exponential L\'evy models, 
with one probabilistic and one ``deterministic'' proof.
The former one is based on concepts from statistical learning theory,
and provides for relevant payoffs 
(baskets, call on max, \ldots) an expression error $\cO(\varepsilon)$ with DNN sizes of
$\cO(\varepsilon^{-2})$ and with constants implied in $\cO(\cdot)$ which grow polynomially in $d$, 
thereby overcoming the curse of dimensionality, whereas the latter is based on
parabolic smoothing of the Kolmogorov equation, and allows us to
prove \emph{exponential expressivity of prices for positive maturities},
i.e.\ an expression error $\cO(\varepsilon)$ with DNN sizes of
$\cO(|\log \varepsilon|^a)$ for some $a>0$, 
albeit with constants implied in $\cO(\cdot)$ 
possibly growing exponentially in $d$.

For the latter approach certain non-degeneracy is required 
on the symbol of the underlying L\'evy process. 
The probabilistic proof of DNN approximation rate results, on the other hand, 
does not require any such assumptions. 
It only relies on the additive structure of the 
semigroup associated to the L\'evy process and existence of moments. 
Thus, the results proved here are specifically tailored to the class of option pricing functions 
(or more generally expectations of exponential L\'evy processes) under European style, plain vanilla payoffs.

The structure of this paper is as follows. 
In Section \ref{sec:ExpLevPIDEs} we review terminology, basic results, and
financial modelling with exponential L\'evy processes. 
In particular, we also recapitulate the corresponding fractional, partial integrodifferential
Kolmogorov equations which generalize the classical Black--Scholes equations 
to L\'evy models.
Section \ref{sec:DNNs} recapitulates notation and basic terminology for deep neural networks
to the extent required in the ensuing expression rate analysis. We focus mainly on so-called
ReLU DNNs, but add that corresponding definitions and also results do hold for more general
activation functions.
In Section \ref{sec:DNNLevy1d} we present a first set of DNN expression rate results, still
in the univariate case. This is, on the one hand, for presentation purposes, as this setting
allows for lighter notation, and to introduce mathematical concepts which will be used
subsequently also for contracts on possibly large basket of L\'evy-driven risky assets. 
We also present an application of the results to neural-network based call option pricing. 
Section \ref{sec:DNNMultivLevy} then has the main results of the present paper: 
expression rate bounds for ReLU DNNs for multivariate, exponential L\'evy models. 
We identify sufficient conditions to obtain expression rates which are free from the curse of 
dimensionality via mathematical tools from statistical learning theory. 
We also develop a second argument based on parabolic Gevrey regularity
with quantified derivative bounds, which even yield exponential expressivity of ReLU DNNs, 
albeit with constants that generally depend on the basket size 
in a possibly exponential way. 
Finally, we develop an argument based on quantified sparsity 
in polynomial chaos expansions and corresponding ReLU expression rates
from Schwab and Zech \cite{SZ19_2592} to prove high algebraic expression rates for ReLU DNNs,
with constants that are independent of the basket size.
We also provide a brief discussion of recent, related
results. 
We conclude in Section \ref{seC:ConclGen} and indicate several possible generalizations of the present results.
%%%%%%%%%%%%%%%%%%%%%%%%%%%%%%%%%%%%%%%%%%%%%%%%%%%%%%%%%%%%%%%%%%%%%%%%%%%%%%%%%%%%%%%%%5
\section{Exponential L\'evy models and PIDEs}
\label{sec:ExpLevPIDEs}
%%%%%%%%%%%%%%%%%%%%%%%%%%%%%%%%%%%%%%%%%%%%%%%%%%%%%%%%%%%%%%%%%%%%%%%%%%%%%%%%%%%%%%%%%5
\subsection{L\'evy processes}
\label{sec:LP}
%%%%%%%%%%%%%%%%%%%%%%%%%%%%%%%%%%%%%%%%%%%%%%%%%%%%%%%%%%%%%%%%%%%%%%%%%%%%%%%%%%%%%%%%%5
Fix a complete probability space $(\Omega,\Fc,\IP)$ 
on which all random elements are defined.

We start with the univariate case.
We recall that an $\mathbb{R}$-valued continuous-time process $(X_t)_{t\geq0}$ 
is called a L\'evy process if it is stochastically continuous, 
it has almost surely RCLL sample paths, 
it satisfies $X_0 = 0$ almost surely, and it has stationary and independent increments. 
See, e.g. Bertoin \cite{BertoinLP1996}, Sato \cite{Sato1999} for discussion and for 
detailed statements of definitions.

It is shown in these references that
a L\'evy process (LP for short) $X$ is characterized by its so-called 
L\'evy triplet $(\sigma^2,\gamma,\nu)$, where $\sigma \geq 0$, $\gamma \in \mathbb{R}$ 
and where $\nu$ is a measure on $(\mathbb{R},\cB(\IR))$ 
with $\nu(\{0\}) = 0$, 
the so-called \emph{jump-measure, or L\'{e}vy-measure} of the LP $X$
which satisfies 
$\int_\mathbb{R} (x^2 \wedge 1)\, \nu (d x) < \infty$. 
For more details on both univariate LPs and 
the multivariate situation we refer to \cite{Sato1999}.

As in the univariate case, multivariate ($\R^d$-valued) LPs $X^d$ 
are completely described by their characteristic triplet 
$(A^d,\gamma^d, \nu^d)$ where $\gamma^d \in \IR^d$ is a drift vector, 
$A^d \in \IR^{d\times d}$ is a symmetric, non-negative definite matrix
denoting the covariance matrix of the Brownian motion part of $X^d$,
and $\nu^d$ is the L\'{e}vy measure describing the jump structure of $X^d$.

To characterize the dependence structure of a L\'evy process 
the drift parameter $\gamma^d$ does not play a role. 
The dependence structure of the diffusion part of $X^d$ is characterized by $A^d$. 
Since the continuous part and
the jump part of $X^d$ are stochastically independent, 
the dependence structure of the jump part of $X^d$ is characterized by the 
L\'{e}vy measure $\nu^d$.

In Kallsen and Tankov \cite{KT06}, a characterization of admissible jump measures $\nu^d$ of the $\R^d$-valued LP $X^d$
has been obtained as \emph{superposition} of marginal, univariate L\'evy measures with
a so-called \emph{L\'{e}vy copula} function.
%%%%%%%%%%%%%%%%%%%%%%%%%%%%%%%%%%%%%%%%%%%%%%%%%%%%%%%%%%%%%%%%%%%%%%%%%%%%%%%%%%%%%%%%%5

\subsection{Exponential L\'evy models}
\label{sec:ExpLP}
%%%%%%%%%%%%%%%%%%%%%%%%%%%%%%%%%%%%%%%%%%%%%%%%%%%%%%%%%%%%%%%%%%%%%%%%%%%%%%%%%%%%%%%%%5
In this article we are interested in 
estimating expression rates of deep neural networks for 
approximating the function 
$s \mapsto \E[\varphi(s S_T)]$, where $S$ is an exponential of a 
$d$-dimensional L\'evy process and $\varphi \colon \R^d \to \R$ an appropriate function. 
The key motivation for studying such expectations comes from the context of option valuation. 
Thus, we now outline this relation and we will always use the language of option pricing, i.e., 
refer to these expectations as option prices and to $\varphi$ as the payoff. 
This interpretation is justified if $S$ is a 
%local (hence true) 
martingale 
and we state below the conditions on the L\'evy process that guarantee this. 
%We emphasize, however, 
%that $(S_t)_{t\geq 0}$
%does not need to be a (local) martingale for the results proved in this article to hold.
%

%In a widely used class of financial models, 
%log-returns of risky assets are modelled by LPs.
%This leads to price processes which are exponential L\'evy models, 
%generalizing geometric Brownian motion.

Let the $\R$-valued stochastic process $(S_t)_{t \in [0,T]}$ 
model the price of one risky financial asset. 
Here $T \in (0,\infty)$ is a fixed, finite time horizon.  
An \emph{exponential L\'evy model} assumes 
that $S_t = S_0 e^{rt + X_t}$, $t \in [0,T]$, 
where $r \in \R$ denotes the (constant) interest rate. 
The model could be specified either under a real-world measure or 
directly under a risk-neutral measure 
(constructed using the general change of measure result in 
\cite[Theorems~33.1 and 33.2]{Sato1999} 
of which the Esscher transform Gerber and Shiu \cite{GerberShiu1994} 
is a particular case, or 
by minimizing certain functionals over the family of 
equivalent martingale measures, see for instance 
Jeanblanc et al.\ \cite{jeanblanc2007}, Esche and Schweizer \cite{ESCHE2005} and the references therein).
The latter situation means that $(S_t e^{-rt})_{t \in [0,T]}$ is a martingale, 
which is equivalent to the following condition on the L\'evy triplet of $X$
(e.g. Hilber et al.\ \cite[Lemma~10.1.5]{CMQFBook})
\begin{equation}\label{eq:martingale1} \gamma = - \frac{\sigma^2}{2}
-  \int_{\R}(e^y -1 - y \, \mathbbm{1}_{\{|y|\leq 1\}})\nu(d y), \quad \int_{\{|y|>1\}} e^{y} \nu(d y) < \infty.
\end{equation}
For a $d$-dimensional L\'evy process $X^d$, \cite[Theorem~25.17]{Sato1999} 
shows that the multivariate geometric L\'evy process
$(e^{X^d_{t,1}},\ldots,e^{X^d_{t,d}})_{t \geq 0}$ 
is a martingale if and only if 
\begin{equation}\label{eq:martingaleD}
\begin{aligned} \int_{\{\|y\|> 1\}} e^{ y_i} \nu^d(d y) < \infty, \quad & \text{for } i=1,\ldots,d,
\\ 
\gamma_i^d = - \frac{A_{ii}^d}{2} - \int_{\R^d} (e^{y_i}-1- y_i \mathbbm{1}_{\{\|y\|\leq 1\}}) \nu^d(d y), \quad & \text{for } i=1,\ldots,d.
\end{aligned}
\end{equation}
This condition ensures that the functions defined in \eqref{eq:optionPrice} and \eqref{eq:optionPriceD} 
below represent option prices. However, the condition is not needed for the proof of the results later, 
so we do not need to impose \eqref{eq:martingale1} or \eqref{eq:martingaleD} 
in any of the results proved in the article. 
We will, however, impose certain moment or regularity conditions.

For more details on exponential L\'evy models, 
with particular attention to their use in financial modelling, 
we refer to Cont and Tankov \cite{Cont2004}, Lamberton and Mikou \cite{LambMik08} and Eberlein and Kallsen \cite{EberKall19} and the references there.
%%%%%%%%%%%%%%%%%%%%%%%%%%%%%%%%%%%%%%%%%%%%%%%%%%%%%%%%%%%%%%%%%%%%%%%%%%%%%%%%%%%%%%%%%5
\subsection{PIDEs for option prices}
\label{sec:PIDEOptPr}
%%%%%%%%%%%%%%%%%%%%%%%%%%%%%%%%%%%%%%%%%%%%%%%%%%%%%%%%%%%%%%%%%%%%%%%%%%%%%%%%%%%%%%%%%5
Let us first discuss the case of a univariate exponential L\'evy model.
For the multivariate case we refer to Section~\ref{sec:DNNMultivLevy} 
(cf.\ \eqref{eq:optionPriceD} and \eqref{eq:MultdPDEs} below).

Consider a European style option with payoff function 
$\varphi \colon (0,\infty) \to [0,\infty)$ and 
at most polynomial ($p$-th order) growth at infinity. 
Assume for this subsection that \eqref{eq:martingale1} is satisfied.

The value of the option (under the chosen risk-neutral measure) 
at time $t \in [0,T]$ 
is given as the conditional expectation 
$C_t =\E[e^{-r(T-t)}\varphi(S_T)|\Fc_t]$ with $\Fc_t = \sigma(S_v:v\in[0,t])$. 
By the Markov property 
$C_t = C(t,S_t)$ and so, 
switching to  time-to-maturity $\tau = T-t$, 
$u(\tau,s)=C(T-\tau,s)$ 
we can rewrite the option price as follows:
\begin{equation}
\label{eq:optionPrice}
u(\tau,s)  
=  \E[e^{-r\tau}\varphi(S_T)|S_t =s]
=  \E\left[e^{-r\tau}\varphi\big(s\exp(r\tau+X_\tau)\big)\right]
\end{equation}
for $\tau \in [0,T]$, $s \in (0,\infty)$,
where the second step uses that $X_T - X_t$ is independent of $X_t$ 
and has the same distribution as $X_{T-t}$. 
If the payoff function $\varphi$ is Lipschitz-continuous on $\R$ 
and the L\'evy process fulfils either $\sigma>0$ or a certain non-degeneracy condition on $\nu$, 
then $u$ is continuous on $[0,T)\times(0,\infty)$, 
it is $C^{1,2}$ on $(0,T)\times(0,\infty)$ 
and it satisfies the 
\emph{linear, parabolic partial integrodifferential equation} (PIDE for short)
\begin{equation}\label{eq:PIDE}
\begin{aligned}
\frac{\partial u}{\partial \tau}(\tau,s) 
& -rs\frac{\partial u}{\partial s}(\tau,s) - \frac{\sigma^2 s^2}{2} \frac{\partial^2 u}{\partial s^2}(\tau,s) - ru(\tau,s) 
\\ 
& - \int_{\R}\left(u(\tau,se^y)-u(\tau,s)-s(e^y-1)\frac{\partial u}{\partial s}(\tau,s)\right) \nu(d y) 
= 0
\end{aligned} 
\end{equation}
on $[0,T)\times(0,\infty)$ with initial condition 
$u(0,\cdot) = \varphi$, see for instance Proposition~2 in Cont and Voltchkova \cite{ContVolt2005}. 
If the non-degeneracy condition on $\nu$ is dropped, 
one can still characterize $u$ (transformed to log-price variables) as 
the unique viscosity solution to the PIDE above. 
This is established e.g. in  \cite{ContVolt2005} (see also Proposition~3.3 in \cite{ContVolt2006}). 
For our purposes the representation \eqref{eq:optionPrice} is more suitable. 
However, by using this characterization 
(also called Feynman--Kac representation for viscosity-solutions of PIDEs, see Barles et al.\ \cite{BBP1997}) 
the results 
formulated below also provide DNN approximations for PIDEs.
Finally, note that the interest rate $r$ may also be directly modelled 
as a part of $X$ by modifying $\gamma$.
To simplify the notation we set $r=0$ in what follows.
We also remark that all expression rate results hold verbatim
for assets with a constant dividend payment 
(see, e.g., \cite[Eqn. (3.1)]{LambMik08} for the 
functional form of the exponential L\'evy model in that case).
%%%%%%%%%%%%%%%%%%%%%%%%%%%%%%%%%%%%%%%%%%%%%%%%%%%%%%%%%%%%%%%%%%%%5
\section{Deep neural networks (DNNs)}
\label{sec:DNNs}
%%%%%%%%%%%%%%%%%%%%%%%%%%%%%%%%%%%%%%%%%%%%%%%%%%%%%%%%%%%%%%%%%%%%%55
This article is concerned with establishing expression rate bounds of deep neural networks (DNNs) for prices of options (and the associated PIDEs) written on possibly large baskets of risky assets, 
whose log-returns are modelled by a multivariate L\'evy process with general correlation structure of jumps. The term ``expression rate'' denotes the rate of convergence to $0$ of the error between the option price and its DNN approximation. This rate can be directly translated to quantify the DNN size required to achieve a given approximation accuracy. For instance, in Theorem~\ref{prop:Ddresult} below an expression rate of $\mathfrak{q}^{-1}$ is established and one may even choose $\mathfrak{q} = 2$ in many relevant cases. We now give a brief introduction to DNNs.

Roughly speaking,
a deep neural network (DNN for short) 
is a function built by multiple concatenations of affine transformations 
with a (typically non-linear) activation function. 
This gives rise to a parametrized family of non-linear maps,
see for example Petersen and Voigtlaender \cite{PETERSEN2018296} or Buehler et al.\ \cite[Section~4.1]{Buehler2018} 
and the references there.

Here we follow current practice and refer to the collection of parameters $\Phi$ 
as \emph{``the neural network''} and 
denote by $\Rl(\Phi)$ its realization, that is, the function defined by these parameters. 
More specifically, we use the following terminology (see for example Section~2 in Opschoor et al.\ \cite{Opschoor2020}): 
firstly, we fix a function $\varrho \colon \R \to \R $ (referred to as the activation function) 
which is applied componentwise to vector-valued inputs.

\begin{definition}
	\label{def:DNN}
	Let $d, L \in \N$. 
	A neural network (with $L$ layers and $d$-dimensional input) is a collection 
	\[
	\Phi = \big((A_1,b_1),\ldots,(A_L,b_L)\big),
	\]
	where 
	$N_0=d$, $N_i \in \N$, $A_i \in \R^{N_i \times N_{i-1}}$, 
	$b_i \in \R^{N_i}$ for $i=1,\ldots,L$ and $(A_i,b_i)$ 
	are referred to as the weights of the $i$-th layer of the NN.
	
	The associated realization of $\Phi$ is the mapping 
	\[
	\begin{aligned} \Rl(\Phi)  \colon  \R^d &\to \R^{N^L}, \quad 
	x & \mapsto \Rl(\Phi)(x)= A_L x_{L-1} + b_L
	\end{aligned}
	\]
	where $x_{L-1}$ is given as
	\[ x_0=x, \quad x_l = \varrho(A_l x_{l-1} + b_l) \text{ for } l=1,\ldots,L-1.
	\]
	We call $M_j(\Phi) = \|A_j\|_0 + \|b_j\|_0$ the number of (non-zero) weights 
	in the $j$-th layer and $M(\Phi) = \sum_{j=1}^L M_j(\Phi)$ the number of weights 
	of the neural network $\Phi$. 
	We also refer to $M(\Phi)$ as the \textit{size} of the neural network, 
	write $L(\Phi) = L$ for the number of layers of $\Phi$ 
	and refer to $N_{o}(\Phi)=N_L$ as the output dimension.
\end{definition}

We refer to Section~2 in Opschoor et al.\ \cite{Opschoor2020} for further details. 

The following lemma shows that concatenating $n$ affine transformations 
with distinct neural networks and taking their weighted average 
can itself be represented as a neural network. 
The number of non-zero weights in the resulting neural network 
can be controlled by the number of non-zero weights in the original neural networks. 
The proof of the lemma is based on a simple extension of the \textit{full parallelization} 
operation for neural networks (see  \cite[Proposition~2.5]{Opschoor2020}) 
and refines Grohs et al.\ \cite[Lemma~3.8]{HornungJentzen2018}.
\begin{lemma} \label{lem:averageNN}
	Let $d,L,n \in \N$ and let $\Phi^1,\ldots,\Phi^n$ be neural networks with $L$ layers, 
	$d$-dimensional input and equal output dimensions. 
	Let $D_1,\ldots,D_n$ be $d\times d$-matrices, 
	$c_1,\ldots,c_n \in \R^d$ and 
	$w_1,\ldots,w_n \in \R$.
	
	Then there exists a neural network $\psi$ such that 
	\begin{equation}\label{eq:averageNN} 
	\Rl(\psi)(x) = \sum_{i=1}^n w_i \Rl(\Phi^i)(D_i x + c_i) \quad \text{ for all } x\in \R^d
	\end{equation}
	and $M_j(\psi) \leq \sum_{i=1}^n M_j(\Phi^i)$ for $j=2,\ldots,L$. 
	If, in addition, $D_1,\ldots,D_n$ are diagonal matrices 
	and $c_1=\cdots=c_n=0$, 
	then $M(\psi) \leq \sum_{i=1}^n M(\Phi^i)$.
\end{lemma}

\begin{proof} 
	Write for $i=1,\ldots,n$
	\[ \Phi^i =\left((A_1^i,b_1^i),\ldots,(A_L^i,b_L^i)\right)
	\]
	and define the block matrices 
	\begin{alignat*}{2}
	A_1^{n+1} &=  \begin{pmatrix}
	A_1^1 D_1 \\
	\vdots    \\
	A_1^n D_n
	\end{pmatrix}, \quad && b_1^{n+1} =  \begin{pmatrix}
	A_1^1 c_1 + b_1^1 \\
	\vdots    \\
	A_1^n c_n + b_1^n
	\end{pmatrix}, 
	\\
	A_j^{n+1} & =  \begin{pmatrix}
	A_j^1 &\multicolumn{2}{c}{\text{\kern0.5em\smash{\raisebox{-1ex}{\LARGE 0}}}}\\
	&  \ddots &   \\
	\multicolumn{2}{c}{\text{\kern-2.0em\smash{\raisebox{-.5ex}{\LARGE 0}}}}& A_j^n
	\end{pmatrix}, \quad   && b_j^{n+1}=  \begin{pmatrix}
	b_j^1 \\
	\vdots    \\
	b_j^n
	\end{pmatrix} \text{ for $j=2,\ldots,L-1$},
	\\
	A_L^{n+1} & =   \begin{pmatrix}
	w_1 A_L^1 &  \cdots  & w_n A_L^n
	\end{pmatrix}  \text{ and } && b_L^{n+1} =  
	w_1 b_L^1+\cdots+
	w_n b_L^n.
	\end{alignat*}
	Set $\psi= ((A_1^{n+1},b_1^{n+1}),\ldots,(A_L^{n+1},b_L^{n+1})) $. Then, for $l\in \{1,\ldots,L-1\}$ and $x \in \R^d$, 
	it is straightforward to verify that $x_l$ has a block structure 
	(with subscripts indicating the layers and superscripts indicating the blocks)
	\[ x_l  =  \begin{pmatrix}
	x_l^1 \\
	\vdots    \\
	x_l^n
	\end{pmatrix},
	\]
	with $x_1^i = \varrho(A_1^i (D_i x + c_i) + b_1^i)$,  $x_l^i = \varrho(A_l^i x_{l-1}^i + b_l^i)$ for $l=2,\ldots,L-1$ and 
	\[
	\Rl(\psi)(x) = A^{n+1}_L x_{L-1} + b^{n+1}_L = \sum_{i=1}^n w_i(A_L^i x_{L-1}^i + b_L^i).
	\]
	Hence, \eqref{eq:averageNN} is satisfied and
	\[ \begin{aligned}
	M_j(\psi) &= M_j(\Phi^1)+\cdots+M_j(\Phi^n) \text{ for $j=2,\ldots,L-1$},
	\\
	M_L(\psi) & \leq M_L(\Phi^1) \mathbbm{1}_{\{w_1 \neq 0\}} +\cdots+  M_L(\Phi^n) \mathbbm{1}_{\{w_n \neq 0\}}.
	\end{aligned}
	\]
	If in addition $D_1,\ldots,D_n$ are diagonal matrices and $c_1=\cdots=c_n=0$, then $\|A_1^i D_i\|_0 = \|A_1^i\|_0$ and therefore $M_1(\psi) = M_1(\Phi^1)+\cdots+M_1(\Phi^n)$. Thus, in this situation, 
	$M(\psi) = \sum_{j=1}^L M_j(\psi) \leq \sum_{j=1}^L \sum_{i=1}^n M_j(\Phi^i) = \sum_{i=1}^n M(\Phi^i)$, as claimed. 
\end{proof}
%%%%%%%%%%%%%%%%%%%%%%%%%%%%%%%%%%%%%%%%%%%%%%%%%%%%%%%%%%%%%%%
\section{DNN approximations for univariate L\'evy models}
\label{sec:DNNLevy1d}
%%%%%%%%%%%%%%%%%%%%%%%%%%%%%%%%%%%%%%%%%%%%%%%%%%%%%%%%%%%%%%%%
We study DNN expression rates for option 
prices under (geometric) L\'evy models for asset prices,
initially here in one spatial dimension. 
We present two expression rate estimates for ReLU DNNs, 
which are based on distinct mathematical arguments: the first, probabilistic 
argument builds on ideas used in 
recent works Gonon et al.\ \cite{Gonon2019}, Beck et al.\ \cite{BGJ20_889} and the references there. 
However, 
for the key step of the proof a different technique is used, which is
based on the Ledoux-Talagrand contraction principle (Theorem~4.12 in Ledoux and Talagrand \cite{Ledoux2013}) 
and statistical learning.
This new approach is not only technically less involved 
(in comparison to, e.g., the techniques used in \cite{Gonon2019}), 
but it also allows for weaker assumptions on the activation function,
see Proposition~\ref{prop:1dresult} below. 
Alternatively, under slightly different hypotheses on the activation function 
one can also rely on \cite[Lemma~2.16]{Gonon2019}, 
see Proposition~\ref{prop:1dresult2} below.
The probabilistic arguments result in, essentially, 
$\eps$-complexity of DNN expression of order $\eps^{-2}$.
The second argument draws on parabolic (analytic) regularity furnished by the corresponding
Kolmogorov equations, and 
results in far stronger, exponential expression rates, i.e., 
with an $\eps$-complexity of DNN expression scaling, essentially, 
polylogarithmic with respect to $0< \eps < 1$.
As we shall see in the next section, however, 
the latter argument is in general subject to the curse of dimensionality.
%%%%%%%%%%%%%%%%%%%%%%%%%%%%%%%%%%%%%%%%%%%%%%%%%%%%%%%%%%%%%%%%%%
\subsection{DNN expression rates: probabilistic argument}
\label{sec:ProbArg}
%%%%%%%%%%%%%%%%%%%%%%%%%%%%%%%%%%%%%%%%%%%%%%%%%%%%%%%%%%%%%%%%%%
We fix $0< a < b < \infty$ and measure the approximation error 
in the uniform norm on $[a,b]$. 
Recall that $M(\Phi)$ denotes the number of (non-zero) weights 
of a neural network $\Phi$ and $\Rl(\Phi)$ is the realization of $\Phi$. 
Consider the following 
exponential integrability condition on the L\'evy measure $\nu$: for some $p\geq 2$,
\begin{equation}
\label{eq:expMoment} 
\int_{\{|y|>1\}} e^{py} \nu(d y) < \infty. 
\end{equation}
Furthermore, for any function $g$ we denote by $\mathrm{Lip}(g)$ the  best Lipschitz constant for $g$.
\begin{proposition} \label{prop:1dresult} Suppose the moment condition \eqref{eq:expMoment} holds.
	Suppose further the payoff $\varphi$ can be approximated 
	by neural networks, 
	that is, given a payoff function $s\mapsto \varphi(s)$ 
	there exists constants $c>0$, $q\geq0$ such that 
	for any $\varepsilon \in (0,1]$ 
	there exists a neural network $\phi_{\varepsilon}$ with  
	\begin{align} \label{eq:NNclose}
	|\varphi(s)-\Rl(\phi_{\varepsilon})(s)| & \leq \varepsilon c (1+|s|), \quad s \in (0,\infty),
	\\ \label{eq:NNsparse} M(\phi_{\varepsilon}) &\leq c \varepsilon^{-q}, 
	\\ \label{eq:NNlipschitz} \mathrm{Lip}(\Rl(\phi_{\varepsilon})) & \leq c.
	\end{align}
	Then there exists $\kappa \in [c,\infty)$ (depending on the interval $[a,b]$)
	and neural networks $\psi_{\varepsilon}$, $\varepsilon \in (0,1]$,  
	such that for any target accuracy $\varepsilon \in (0,1]$ 
	the number of weights is bounded by 
	$M(\psi_{\varepsilon}) \leq \kappa \varepsilon^{-2-q}$ 
	and 
	the approximation error between the neural network $\psi_{\varepsilon}$ 
	and the option price is at most $\varepsilon$, that is, 
	\[
	\sup_{s \in [a,b]} |u(T,s) - \Rl(\psi_{\varepsilon})(s)| \leq \varepsilon.
	\]
\end{proposition}

\begin{remark} \label{rmk:q=0}
	In relevant examples such as, e.g., plain vanilla European options,
	the initial condition can be represented \textit{exactly} as a neural network $\phi$. 
	Then one can choose $\phi_{\varepsilon}=\phi$ for all $\varepsilon \in (0,1]$ 
	and so \eqref{eq:NNclose}--\eqref{eq:NNlipschitz} is satisfied with $q=0$, 
	$c=\max\{M(\phi),\mathrm{Lip}(\Rl(\phi))\}$. 
	Examples include call options, straddles, and butterfly payoff functions 
	(when $\varrho$ is the ReLU activation function given by $x \mapsto \max\{x,0\}$).
\end{remark}
\begin{remark}
	In Proposition~\ref{prop:1dresult} the time horizon $T>0$ is finite and fixed. 
	As evident from the proof, the constant $\kappa$ depends on $T$. 
\end{remark}
\begin{proof} 
	Let $\varepsilon \in (0,1]$ be the given target accuracy and fix
	$\bar{\varepsilon} \in (0,1]$ (to be specified later). 
	Denote $\phi = \phi_{\bar{\varepsilon}}$.
	First, \eqref{eq:NNclose} and \eqref{eq:NNlipschitz} 
	show for any $s \in (0,\infty)$ that 
	\[ \begin{aligned}
	|\varphi(s)|&\leq |\varphi(s)-\Rl(\phi)(s)|+|\Rl(\phi)(s)-\Rl(\phi)(0)| + |\Rl(\phi)(0)|
	\\ & \leq  \bar{\varepsilon} c (1+|s|) + c|s|+|\Rl(\phi)(0)|.
	\end{aligned}
	\]
	Thus, 
	$\varphi$ is at most linearly growing at $\infty$. 
	Hence we obtain $\E[\varphi(se^{X_T})] < \infty$, 
	since even the second exponential moment is finite, i.e.,
	\begin{equation}
	\label{eq:expMoment2} 
	\E[e^{2 X_T}] < \infty,
	\end{equation}
	due to the assumed integrability \eqref{eq:expMoment} 
	of the L\'evy measure and Sato \cite[Theorem~25.17]{Sato1999}.
	
	Now recall that 
	\[ 
	u(T,s) = \E[\varphi(se^{X_T})]. 
	\]
	Combining this with assumption \eqref{eq:NNclose} yields for all $s \in [a,b]$
	\begin{align}\nonumber 
	|u(T,s) - \E[\Rl(\phi)(se^{X_T})]| 
	& \leq \E[|\varphi(se^{X_T}) - \Rl(\phi)(se^{X_T})|] 
	\\ & \leq \bar{\varepsilon} c (1+|s|\E[e^{ X_T}]) 
	\leq \bar{\varepsilon} c_1 \label{eq:auxEq4}
	\end{align}
	with the constant $c_1 = c (1+b\E[e^{ X_T}])$ being finite due to \eqref{eq:expMoment2}.
	
	In the second step, let $X^1,\ldots,X^n$ denote $n$ i.i.d.\ copies of $X$ 
	and introduce an independent collection of Rademacher random variables 
	$\varepsilon_1,\ldots,\varepsilon_n$. Write $f(s) = \Rl(\phi)(s)-\Rl(\phi)(0)$ for $s \in (0,\infty)$. 
	Note that the mapping 
	$
	\R^n \times \R^n \ni (x,y) \mapsto \sup_{s \in [a,b]} \left| \sum_{k=1}^n y_k f(s e^{x_k}) \right| 
	$ is Borel-measurable, because the supremum over $s \in [a,b]$ equals 
	the supremum over $s \in [a,b] \cap \mathbbm{Q}$ due to continuity of $f$ and 
	the pointwise supremum of a countable collection of measurable functions is itself measurable. 
	The same reasoning guarantees that the suprema over $s \in [a,b]$ in \eqref{eq:auxEq50}-\eqref{eq:auxEq1} 
	below are indeed random variables, 
	because they are equal to the respective suprema over $s \in [a,b] \cap \mathbbm{Q}$.
	
	Using independence and a standard symmetrization argument
	(see for example Boucheron et al.\ \cite[Lemma~11.4]{Boucheron2013}) 
	we obtain
	\begin{align} \nonumber
	& \E\left[\sup_{s \in [a,b]} \left| 
	\E[\Rl(\phi)(se^{X_T})] - \frac{1}{n} \sum_{k=1}^n \Rl(\phi)(se^{X_T^k}) \right| \right]  
	\\ \label{eq:auxEq50} 
	& \leq 2
	\E\left[\sup_{s \in [a,b]} \left| \frac{1}{n} \sum_{k=1}^n \varepsilon_k f(se^{X_T^k}) \right| \right] . 
	\end{align}
	%\rev{Using this} in the first step, 
	Elementary properties of conditional expectations in the first step 
	and Theorem~4.12 in Ledoux and Talagrand \cite{Ledoux2013} 
	(with $T$ in that result chosen as $T_{x_1,\ldots,x_n}=\{t \in \R^n \colon t_1=se^{x_1},\ldots,t_n=se^{x_n} \text{ for some } s \in [a,b] \}$) 
	in the second step show that
	\begin{align} \nonumber
	& 2 \E\left[\sup_{s \in [a,b]} \left| \frac{1}{n} \sum_{k=1}^n \varepsilon_k f(se^{X_T^k}) \right| \right]  
	\\ \nonumber & = \frac{2}{n}
	\E\left[ \left.\E\left[\sup_{t \in T_{x_1,\ldots,x_n}} 
	\left| \sum_{k=1}^n \varepsilon_k f(t_k) \right|\right]\right|_{x_1=X^1_T,\ldots,x_n=X^n_T} \right]
	\\ \nonumber & \leq \frac{4}{n} \mathrm{Lip}(\Rl(\phi))
	\E\left[ \left.\E\left[\sup_{t \in T_{x_1,\ldots,x_n}} 
	\left| \sum_{k=1}^n \varepsilon_k t_k \right|\right]\right|_{x_1=X^1_T,\ldots,x_n=X^n_T} \right]    
	\\\nonumber & = \frac{4}{n} \mathrm{Lip}(\Rl(\phi))
	\E\left[\sup_{s \in [a,b]} \left| \sum_{k=1}^n \varepsilon_k se^{X_T^k} \right| \right] 
	\\ & \leq \frac{4 b }{n} \mathrm{Lip}(\Rl(\phi)) \E\left[\left| \sum_{k=1}^n \varepsilon_k e^{X_T^k} \right| \right].\label{eq:auxEq1} \end{align}
	On the other hand, one may apply Jensen's inequality, independence and $\E[\varepsilon_k \varepsilon_l]=\delta_{k,l}$ to estimate
	\[\begin{aligned} \E\left[\left| \sum_{k=1}^n \varepsilon_k e^{X_T^k} \right| \right] & \leq \E\left[\left| \sum_{k=1}^n \varepsilon_k e^{X_T^k} \right|^2 \right]^{1/2} = \left(\sum_{k=1}^n \E\left[e^{2 X_T^k} \right]\right)^{1/2} \\ & = \sqrt{n} \E\left[e^{2 X_T} \right]^{1/2}. \end{aligned}
	\]
	Combining this with the previous estimates \eqref{eq:auxEq50}-\eqref{eq:auxEq1}  and the hypothesis on the Lipschitz-constant of the neural network \eqref{eq:NNlipschitz} we obtain that 
	\begin{equation}\label{eq:auxEq2}
	\E\left[\sup_{s \in [a,b]} \left| \E[\Rl(\phi)(se^{X_T})] - \frac{1}{n} \sum_{k=1}^n \Rl(\phi)(se^{X_T^k}) \right| \right]  \leq \frac{c_2}{\sqrt{n}}
	\end{equation}
	with $c_2 = 4 b c \E[e^{2 X_T}]^{1/2}$, which is finite again due to the existence of exponential moments \eqref{eq:expMoment2}.
	
	In a third step we can now apply Markov's inequality for the first estimate and then insert \eqref{eq:auxEq2} to estimate
	\begin{align}\nonumber
	\P & \left[ \sup_{s \in [a,b]} \left| \E[\Rl(\phi)(se^{X_T})] - \frac{1}{n} \sum_{k=1}^n \Rl(\phi)(se^{X_T^k}) \right| \geq \frac{3 c_2}{2\sqrt{n}} \right] \\ \nonumber & \leq \frac{2\sqrt{n}}{3 c_2} \E\left[\sup_{s \in [a,b]} \left| \E[\Rl(\phi)(se^{X_T})] - \frac{1}{n} \sum_{k=1}^n \Rl(\phi)(se^{X_T^k}) \right| \right]
	\\ & \leq \frac{2}{3}.
	\label{eq:auxEq14} \end{align}
	This proves in particular that
	\begin{equation*}
	\P \left[ \sup_{s \in [a,b]} \left| \E[\Rl(\phi)(se^{X_T})] - \frac{1}{n} \sum_{k=1}^n \Rl(\phi)(se^{X_T^k}) \right| \leq \frac{2 c_2}{\sqrt{n}} \right] > 0
	\;.
	\end{equation*}
	Therefore (as $A \in \Fc$ with $\P[A]>0$ necessarily needs to satisfy $A \neq \emptyset$)
	there exists $\omega \in \Omega$ with
	\begin{equation}\label{eq:auxEq16}
	\sup_{s \in [a,b]} 
	\left| \E[\Rl(\phi)(se^{X_T})] - \frac{1}{n} 
	\sum_{k=1}^n \Rl(\phi)(se^{X_T^k(\omega)}) \right| 
	\leq \frac{2 c_2}{\sqrt{n}}.
	\end{equation}
	Lemma~\ref{lem:averageNN} proves that 
	$s \mapsto \frac{1}{n} \sum_{k=1}^n \Rl(\phi)(se^{X_T^k(\omega)}) $ 
	is itself the realization of a neural network $\tilde{\psi}$ with $M(\tilde{\psi}) \leq n M(\phi)$ and hence we have proved the existence of a neural network $\tilde{\psi}$ with
	\begin{equation}\label{eq:auxEq3}
	\sup_{s \in [a,b]} \left| \E[\Rl(\phi)(se^{X_T})] -  \Rl(\tilde{\psi})(s) \right| \leq \frac{2 c_2}{\sqrt{n}}.
	\end{equation}	
	The final step consists in selecting $\bar{\varepsilon} = \varepsilon (c_1+1)^{-1}$, choosing $n = \lceil(2 c_2 \bar{\varepsilon}^{-1})^2 \rceil$, setting $\psi_{\varepsilon} = \tilde{\psi}$, noting (with $\kappa = c(1+4 c_2^2)(c_1+1)^{2+q}$)
	\[
	M(\psi_{\varepsilon}) = M(\tilde{\psi}) \leq n M(\phi) \leq (1+(2 c_2 \bar{\varepsilon}^{-1})^2) c \bar{\varepsilon}^{-q} \leq c(1+4 c_2^2) \bar{\varepsilon}^{-2-q} = \kappa \varepsilon^{-2-q}
	\]
	and combining \eqref{eq:auxEq3} with \eqref{eq:auxEq4} to estimate
	\[\begin{aligned}
	\sup_{s \in [a,b]} &|u(T,s) - \Rl({\psi}_{\varepsilon})(s)| \\ & \leq \sup_{s \in [a,b]} |u(T,s) -  \E[\Rl(\phi)(se^{X_T})]| + \sup_{s \in [a,b]} | \E[\Rl(\phi)(se^{X_T})]- \Rl(\tilde{\psi})(s)|
	\\ & \leq \bar{\varepsilon} (c_1+1) = \varepsilon.
	\end{aligned}\] 
\end{proof}

\begin{proposition} \label{prop:1dresult2}
	Consider the setting of Proposition~\ref{prop:1dresult}, 
	but instead of \eqref{eq:NNlipschitz} 
	assume that $\Rl(\phi_{\varepsilon})$ is $C^1$ and 
	there is a constant $c>0$ such that for every $s\in (0,\infty)$ holds
	\begin{equation} \label{eq:NNgrowth} 
	|\Rl(\phi_{\varepsilon})'(s)| \leq c.
	\end{equation}
	Then the assertion of Proposition~\ref{prop:1dresult} remains valid.
\end{proposition}
\begin{proof}
	This result is a corollary of Proposition~\ref{prop:1dresult}. For the ease of the reader we provide an alternative proof. 
	First, let us verify that \eqref{eq:NNgrowth} and \eqref{eq:NNclose} yield a linear growth condition for 
	$\Rl(\phi_{\varepsilon})$. 
	Indeed, we may use the triangle inequality to estimate for any $\varepsilon \in (0,1]$, $s \in (0,\infty)$, 
	\begin{align}\nonumber
	|\Rl(\phi_{\varepsilon})(s)| & \leq |\Rl(\phi_{\varepsilon})(s)-\Rl(\phi_{\varepsilon})(0)| + |\Rl(\phi_{\varepsilon})(0)-\varphi(0)| + |\varphi(0)| \\ &  \label{eq:NNlineargrowth} \leq \max\{c,|\varphi(0)|\} (1+|s|).
	\end{align}
	Now the same proof as for Proposition~\ref{prop:1dresult} applies, 
	only the second step needs to be adapted. 
	In other words, 
	we prove the estimate \eqref{eq:auxEq2} with a different constant 
	$c_2$ by using a different technique.
	
	To do this, again we let $X^1,\ldots,X^n$ denote $n$ i.i.d.\ copies of $X$. 
	Applying Lemma~2.16 in \cite{Gonon2019} 
	(with random fields $\xi_k(s,\omega)=\Rl(\phi)(se^{X_T^k(\omega)})$, $k=1,\ldots,n$, 
	which satisfy the hypotheses of Lemma~2.16 in \cite{Gonon2019} 
	thanks to \eqref{eq:expMoment2} and \eqref{eq:NNgrowth}) in the first inequality and 
	using \eqref{eq:NNgrowth} and \eqref{eq:NNlineargrowth} for the second inequality then proves that
	\begin{equation*} \begin{aligned} 
	\E&\left[\sup_{s \in [a,b]} \left| \E[\Rl(\phi)(se^{X_T})] - \frac{1}{n} \sum_{k=1}^n \Rl(\phi)(se^{X_T^k}) \right| \right] \\ &  \leq \frac{32 \sqrt{e}}{\sqrt{n}} \sup_{s \in [a,b]} \left( \E[|\Rl(\phi)(se^{X_T})|^2]^{1/2} +(b-a) \E[|\Rl(\phi)'(se^{X_T})e^{X_T}|^2]^{1/2} \right) \\
	&  \leq \frac{32 \max\{c,|\varphi(0)|\} \sqrt{e}}{\sqrt{n}}  \left( 1+b\E[e^{2 X_T}]^{1/2} +(b-a) \E[e^{2 X_T}]^{1/2} \right),
	\end{aligned}\end{equation*}
	which is a bound as in \eqref{eq:auxEq2} with constant 
	\[c_2  = 32 \max\{c,|\varphi(0)|\} \sqrt{e} \left( 1+b\E[e^{2 X_T}]^{1/2} +(b-a) \E[e^{2 X_T}]^{1/2} \right). \]
\end{proof}

\begin{remark}\label{rem:architecture}
	The architecture of the neural network approximations constructed 
	using probabilistic arguments in Proposition~\ref{prop:1dresult}, 
	Proposition~\ref{prop:1dresult2} and also Theorem~\ref{prop:Ddresult} 
	ahead differ from architectures obtained by analytic arguments, 
	see Proposition~\ref{prop:Holom} and Theorem~\ref{thm:DNNExprd} ahead. 
	While the neural networks in the latter results are deep in any situation, 
	the architecture of the neural networks in the former situation depends heavily 
	on the architecture of the neural network $\phi_{\varepsilon}$ 
	used to approximate the payoff function $\varphi$. 
	Therefore, in certain simple situations, 
	the approximating neural network $\psi_{\varepsilon}$ may be a shallow neural network, 
	that is, a neural network with only $L=2$ layers.
	E.g., by \eqref{eq:auxEq4} or \eqref{eq:optionPrice} the function $\varphi$ is specified in the 
	variable $s>0$, and not in log-return variable $x$. 
	This implies, e.g., for a plain-vanilla European call 
	that $\varphi(s) = (s-K)^+$ must be emulated by a ReLU NN, 
	which can be done using the simple $2$-layer neural network
	$\phi_0 = ((1,-K),(1,0))$, that is, $\Rl(\phi_0)=\varphi$.
\end{remark}
%%%%%%%%%%%%%%%%%%%%%%%%%%%%%%%%%%%%%%%%%%%%%%%%%%%%%%%%%%%%%%%%%%%
\subsection{DNN expression of European calls}
\label{sec:DNNCall}
%%%%%%%%%%%%%%%%%%%%%%%%%%%%%%%%%%%%%%%%%%%%%%%%%%%%%%%%%%%%%%%%%%%
In this section we illustrate how the results of Proposition~\ref{prop:1dresult} 
can be used to bound 
DNN expression rates of call options on exponential L\'evy models.

Suppose we observe call option prices for a fixed maturity $T$ and $N$ different 
strikes $K_1,\ldots,K_N>0$. 
Denote these prices by $\hat{C}(T,K_1),\ldots,\hat{C}(T,K_N)$. 
A task frequently encountered in practice is 
to extrapolate from these prices to prices 
corresponding to unobserved maturities or to learn a non-parametric option pricing function.
A widely used approach is to solve 
\begin{equation}
\label{eq:nonparametricPricing}
\min_{\phi \in \mathcal{H}} 
\frac{1}{N} 
\sum_{i=1}^N \left(\frac{\hat{C}(T,K_i)}{K_i}-\phi(S_0/K_i)\right)^2\;.
\end{equation}
Here $\mathcal{H}$ is a suitable collection of (realizations of) neural networks, 
for example all networks with an a-priori fixed architecture. 
In fact, many of the papers listed in the recent review Ruf and Wang \cite{Ruf2020} 
use this approach or a variation of it, where 
for example an absolute value is inserted instead of a square 
or $\hat{C}(T,K_i)/K_i$ is replaced by $\hat{C}(T,K_i)$ and $S/K_i$ by $K_i$. 

In this section we assume that the observed call prices are generated from an (assumed unknown) 
exponential L\'evy model and $\mathcal{H}$ consists of ReLU networks.
Then we show that 
the error in \eqref{eq:nonparametricPricing} can be controlled and 
we can give bounds on the number of non-zero parameters of the minimizing neural network.  
The following result is a direct consequence of Proposition~\ref{prop:1dresult}. 
It shows that $\cO(\varepsilon^{-1})$ weights suffice to achieve an error of 
at most $\varepsilon$ in \eqref{eq:nonparametricPricing}.

\begin{proposition}\label{prop:learningCallPrices} Assume that 
	\[
	\hat{C}(T,K_i) = \E[(S_T-K_i)^+], \text{ for }i=1,\ldots,N,
	\]
	with $S_T = S_0 \exp(X_T)$ and $X$ an (unknown) L\'evy process satisfying \eqref{eq:expMoment}. 
	For any $\kappa >0$, $\varepsilon \in (0,1]$ we let $\mathcal{H}_{\kappa,\varepsilon}$ 
	denote the set of all (realizations of) neural networks with at most $\kappa \varepsilon^{-1}$ non-zero weights
	and choose $\varrho(x)=\max\{x,0\}$ as activation function.
	Then there exists $\kappa \in (0,\infty)$ such that for all $\varepsilon \in (0,1]$
	\[\min_{\phi \in \mathcal{H}_{\kappa,\varepsilon}} \frac{1}{N} \sum_{i=1}^N \left(\frac{\hat{C}(T,K_i)}{K_i}-\phi(S_0/K_i)\right)^2  \leq \varepsilon. \]
\end{proposition}

\begin{proof}
	First, 
	choose the interval $[a,b]$ by setting $a = \min \{S_0/K_1,\ldots,S_0/K_N\}$ 
	and $b=\max \{ S_0/K_1, $ $\ldots, S_0/K_N\}$. 
	We note that the function $\varphi(s)=(s-1)^+$ can be represented by the $2$-layer neural network
	$\phi_0 = ((1,-1),(1,0))$, that is, $\Rl(\phi_0)=\varphi$. Thus, Proposition~\ref{prop:1dresult} 
	can be applied (with $\phi_{\varepsilon}=\phi_0$ for all $\varepsilon \in (0,1]$ and $q=0$, $c=3$)
	and so there exists $\kappa \in [3,\infty)$ 
	and neural networks $\psi_{\delta}$, $\delta \in (0,1]$,  
	such that for any  $\delta \in (0,1]$ 
	we have
	$M(\psi_{\delta}) \leq \kappa \delta^{-2}$ 
	and 
	\[
	\sup_{s \in [a,b]} |u(T,s) - \Rl(\psi_{\delta})(s)| \leq \delta
	\]
	with $u(T,s)=\E[(s e^{X_T}-1)^+]$.
	Therefore, 
	\[\begin{aligned}
	\frac{1}{N} &  \sum_{i=1}^N \left(\frac{\hat{C}(T,K_i)}{K_i}-\Rl(\psi_{\delta})(S_0/K_i)\right)^2 \\ &  = \frac{1}{N} \sum_{i=1}^N \big(u(T,S_0/K_i)-\Rl(\psi_{\delta})(S_0/K_i)\big)^2 
	\leq  \delta^2.
	\end{aligned}
	\]
	Setting $\varepsilon = \delta^2$ and noting $\Rl(\psi_{\delta}) \in \mathcal{H}_{\kappa,\varepsilon}$ then finishes the proof.
\end{proof}

\begin{remark}
	The proof shows that $\kappa$ is independent of $N$. This can also be seen by observing
	that the result directly generalizes to an infinite number of call options with strikes in a compact interval $\mathcal{K} = [\underline{K},\overline{K}]$ with $ \underline{K}> 0$, $ \overline{K}< \infty$. 
	Indeed, let $\mu$ be a probability measure on $(\mathcal{K},\mathcal{B}(\mathcal{K}))$, 
	then choosing $\psi_{\delta}$, $\delta = \varepsilon^2$ as in the proof of 
	Proposition~\ref{prop:learningCallPrices} and 
	$a = S_0/\overline{K}$, $b= S_0 / \underline{K}$  yields $\Rl(\psi_{\delta}) \in \mathcal{H}_{\kappa,\varepsilon}$ 
	and 
	\[ \begin{aligned}
	& \int_{\mathcal{K}}  
	\left(\frac{\hat{C}(T,K)}{K}-\Rl(\psi_{\delta})(S_0/K)\right)^2 \mu(d K) 
	\\  = &
	\int_{\mathcal{K}}  \big(u(T,S_0/K)-\Rl(\psi_{\delta})(S_0/K)\big)^2 \mu(d K) 
	\leq 
	\varepsilon.
	\end{aligned} \]
\end{remark}

%%%%%%%%%%%%%%%%%%%%%%%%%%%%%%%%%%%%%%%%%%%%%%%%%%%%%%%%%%%%%%%%%
\subsection{ReLU DNN exponential expressivity}
\label{sec:exp}
%%%%%%%%%%%%%%%%%%%%%%%%%%%%%%%%%%%%%%%%%%%%
We now develop a second argument for bounding the expressivity
of ReLU DNNs for the option price $u(\tau,s)$ solution of \eqref{eq:PIDE},
subject to the initial condition $u(0,s) = \varphi(s)$. 
In particular, in this subsection we choose $\varrho$ given by $\varrho(x)=\max\{x,0\}$ as activation function.

As in the preceding, probabilistic argument, we consider 
the DNN expression error in a bounded interval $[a,b]$ with $0<a<s<b<\infty$. 
The second argument is based on \emph{parabolic smoothing} of the 
linear, parabolic PIDE \eqref{eq:PIDE}.
This, in turn, ensures smoothness of $s\mapsto u(\tau,s)$ 
at positive times $\tau>0$, i.e.\ 
smoothness in the ``spatial'' variables $s\in [a,b]$ 
resp.\ in the log-return variable $x=\log(s)\in[\log(a),\log(b)]$,
even for non-smooth payoff functions $\varphi$ 
(so, in particular, binary options with discontinuous payoffs $\varphi$ 
are admissible, albeit 
at the cost of non-uniformity of derivative bounds at $\tau \downarrow 0$).
It is a classical result that
this implies spectral, possibly exponential convergence of 
\emph{polynomial approximations} of 
$u(\tau,\cdot)|_{[a,b]}$ in $L^\infty([a,b])$.
As we observed in Opschoor et al.\ \cite[Section 3.2]{OSZ19_839},
this exponential polynomial convergence rate
implies also exponential expressivity of ReLU DNNs of 
$u(\tau,\cdot)|_{[a,b]}$ in $L^\infty([a,b])$ for any $\tau>0$. 

To ensure smoothing properties of the solution operator of the
PIDE, we require additional assumptions 
(see \eqref{eq:StrEllPsi} below) on the L\'evy triplet $(\sigma^2,\gamma,\nu)$.
To formulate these, 
we recall the L\'evy symbol $\psi$ of the $\R$-valued LP $X$
\begin{equation}
\psi(\xi) 
= \frac{\sigma^2}{2} \xi^2 - i\gamma \xi
- \int_{\R} 
\left(e^{i \xi x}-1- i \xi x \mathbbm{1}_{\{|x|\leq 1\}} \right)\nu(dx), \quad \xi \in \R 
\;.
\label{eq:LK2}
\end{equation}
\begin{proposition}\label{prop:Holom}
	Suppose that the symbol $\psi$ of the LP $X$ is such that 
	there exists $\rho \in (0,1]$ and constants $C_i > 0$, 
	$i=1,2,3$ such that for all $\xi \in \IR$ holds
	\begin{equation}\label{eq:StrEllPsi}
	\Re \psi(\xi) \geq C_1 | \xi |^{2\rho}, 
	\quad 
	| \psi(\xi) | \leq C_2 | \xi |^{2\rho} + C_3 \;.
	\end{equation}
	Then, for every $v_0$ such that
	$v_0 = \varphi \circ \exp \in L^2(\IR)$, 
	for every $0< \tau \leq T < \infty$, 
	for every $0<a<b<\infty$, 
	and for every $0<\eps<1/2$ 
	exist neural networks $\psi^{u}_\eps$ 
	which express the solution 
	$u(\tau, \cdot)|_{[a,b]}$ to accuracy $\eps$, i.e.,
	$$
	\sup_{s \in [a,b]}| u(\tau, s) - \Rl(\psi^{u}_\eps)(s) | 
	\leq \eps\;.
	$$
	Furthermore, 
	there exists a constant $C' > 0$ such that 
	with $\delta = \frac{1}{\min\{ 1, 2 \rho \}} \geq 1$ holds
	\[
	M(\psi^{u}_\eps) 
	\leq C' 
	|\log(\eps)|^{2\delta} 
	\;, \;\;\;\; 
	L(\psi^{u}_\eps) \leq C' |\log(\eps)|^{\delta} |\log(|\log(\eps)|)|
	\;.
	\]
\end{proposition}
\begin{remark}
	A sufficient condition on the L\'evy triplet which ensures \eqref{eq:StrEllPsi} is
	as follows.
	Let $X$ be a L\'evy process with characteristic triplet
	$(\sigma^2,\gamma,\nu)$ and L\'evy density $k(z)$ where 
	$\nu(d z) = k(z)d z$ satisfies
	\begin{enumerate}
		\item There are constants $\beta_- >0$, $\beta_+ >1$ and $C > 0$ such that
		\begin{equation*} \label{eq:semiheavy}
		k(z) \le C \begin{cases} e^{-\beta_- \abs{z}},  & \mbox{$z<-1$,}\\
		e^{-\beta_+ z}, & \mbox{$z>1$.}
		\end{cases}
		\end{equation*}
		\item 
		Furthermore, there exist constants $0<\alpha<2$ and $C_+ > 0$ such that
		\begin{equation*} \label{eq:uppbound}
		k(z)  \le C_+ \frac{1}{\abs{z}^{1+\alpha}}, \quad 0<\abs{z}<1.
		\end{equation*}
		\item If $\sigma = 0$, we assume additionally that there is a $C_- > 0$ such that
		\begin{equation*} \label{eq:lowbound}
		\frac{1}{2}\big(k(z)+k(-z)\big) \ge C_- \frac{1}{\abs{z}^{1+\alpha}}, \quad 0<\abs{z}<1.
		\end{equation*}
	\end{enumerate}
	Then \eqref{eq:StrEllPsi} is satisfied  (see \cite[Lemma 10.4.2]{CMQFBook}).
	Here, $\rho = 1$ if $\sigma > 0$ 
	and otherwise 
	$\rho = \alpha/2$. 
\end{remark}
\begin{proof}
	The proof proceeds in several steps: 
	first, we apply the change of variables $x=\log(s) \in \IR$ 
	in order to 
	leverage the stationarity of the LP $X$ for obtaining 
	a constant coefficient Kolmogorov PIDE. 
	Assumptions \eqref{eq:StrEllPsi} then ensure well-posedness
	of the PIDE in a suitable variational framework.
	We then exploit that
	stationarity of the LP $X$ facilitates the use of 
	Fourier transformation;
	the lower bound on $\psi$ in \eqref{eq:StrEllPsi} 
	will allow to derive sharp, explicit 
	bounds on high spatial derivatives
	of (variational) solutions of the PIDE 
	which imply Gevrey regularity of these solutions
	on bounded intervals $[a,b]\subset (0,\infty)$.
	We recall that for $\delta \geq 1$, 
	a smooth function $x\mapsto f(x)$ is Gevrey-$\delta$ regular in an open subset 
	$D\subset \IR^d$ if $f\in C^\infty(D)$ and if
	for every compact set $\kappa \subset\subset D$ exists $C_\kappa > 0$ 
	such that for all $\alpha\in \IN_0^d$ and for every $x\in \kappa$ 
	holds $|D_x^\alpha f(x)| \leq C_\kappa^{|\alpha|+1} (\alpha !)^\delta$.
	Note that $\delta=1$ implies that $f$ is real analytic in $\kappa$.
	We refer to Rodino \cite[Section 1.4]{Rodino} for details, 
	examples and further references. 
	
	Gevrey regularity, in turn, implies exponential rates of convergence
	of polynomial and deep ReLU NN approximations of 
	$s\mapsto u(\tau,s)$ for $\tau>0$ whence we obtain the assertion of the theorem.
	
	We change coordinates to 
	$x = \log(s) \in (-\infty,\infty)$ so that $v(\tau,x)=u(\tau,e^x)$.
	Then, 
	the PIDE \eqref{eq:PIDE} takes the form 
	(e.g.\ Matache et al.\ \cite[Section 3]{MvS04_373}, Lamberton and Mikou \cite[Section 3.1]{LambMik08}) 
	\begin{equation}\label{eq:1}
	\frac{\partial v}{\partial \tau} - \frac{\sigma^2}{2} \frac{\partial^2 v}{\partial x^2}
	- (\gamma + r  )\frac{\partial v}{\partial x}
	+ A[v] + r v = 0\quad\mbox{in }(0, T)\times \IR
	\end{equation}
	where
	$A$ denotes the integrodifferential operator
	\[
	A[f](x) = -
	\int_{\IR} \left( f(x+y) - f(x) - y f '(x) \mathbbm{1}_{\{|y|\leq 1\}} \right) \; \nu(d y) 
	\]
	together with the initial condition
	\begin{equation}\label{eq:2}
	v|_{\tau = 0} = \varphi(e^x) = (\varphi\circ\exp)(x).
	\end{equation}

	Then $C(t, s)= v(T-t, \ln(s))$ satisfies
	\begin{equation}\label{FeynmannKac}
	C(t, S_t) = \E[e^{r(t-T)} \varphi(S_T)| {\mathcal{F}}_t].
	\end{equation}
	Conversely, if $C(t,s)$ in (\ref{FeynmannKac})
	is sufficiently regular, then $v(\tau, x)  \rev{=} C(T-\tau, e^x)$
	is solution of (\ref{eq:1}), (\ref{eq:2}) 
	(recall that we assume $r=0$ for notational simplicity).
	
	The  L{\'e}vy-Khintchine formula describes
	the $\R$-valued LP $X$ by the 
	log-characteristic function $\psi$ of the RV $X_1$.
	From the time-homogeneity of the LP $X$, 
	\begin{equation}\label{LK1}
	\forall t>0: \qquad \E[e^{i\xi X_t}] = e^{-t\psi(\xi)}\;.
	\end{equation}
	The L\'evy exponent $\psi$ of the LP $X$ 
	admits the explicit representation 
	\eqref{eq:LK2}.
	
	The L\'evy exponent $\psi$ is
	the symbol of the pseudo-differential
	operator $-{\mathcal{L}}$, where ${\mathcal{L}}$ is the infinitesimal
	generator of the semi-group
	of the LP $X$.  ${\mathcal{A}}=-{\mathcal{L}}$
	is the spatial operator in \eqref{eq:1} 
	given by
	\begin{equation}\label{calA}
	{\mathcal{A}}[f](x) 
	= 
	- \frac{\sigma^2}{2} \frac{d^2 f}{dx^2}(x) 
	- \gamma
	\frac{df}{dx}(x)  + A[f](x).
	\end{equation}
	For $f, g \in C^\infty_0(\IR)$ we associate with
	operator ${\mathcal{A}}$ the bilinear form
	\begin{equation*}\label{defbilform}
	a(f, g) = \int_{\IR} {\mathcal{A}}[f](x) g(x) dx.
	\end{equation*}
	The translation invariance of the operator $\mathcal A$ (implied by stationarity of 
	the LP $X$) in \eqref{calA} and Parseval's equality (see \cite[Remark 10.4.1]{CMQFBook})
	imply that $\psi$ is the symbol of $\mathcal A$, i.e.\
	$$
	\forall f,g\in C^\infty_0(\IR): 
	\quad 
	a(f,g) 
	= 
	\int_{\IR} \psi(\xi) \hat{f}(\xi) \overline{\hat{g}(\xi)} d\xi
	\;.
	$$
	The assumption \eqref{eq:StrEllPsi} on $\psi$ 
	implies continuity and coercivity
	of the bilinear form $a(\cdot,\cdot)$ on $H^{\rho/2}(\IR)\times H^{\rho/2}(\IR)$,
	so that for $v_0\in L^2(\IR)$ there exists a unique variational solution  $v\in C([0,T];L^2(\IR))\cap L^2(0,T;H^{\rho/2}(\IR))$
	of the PIDE \eqref{eq:1} with the initial condition \eqref{eq:2}, see, e.g.\ Eberlein and Glau \cite{EberleinGlauVarSlnLevyPIDEs}.
	
	Fix $0<\tau\leq T < \infty$, $x\in \IR$ arbitrary. 
	The variational solution $v$ of \eqref{eq:1}, \eqref{eq:2}
	satisfies
	\[
	\begin{aligned} 
	v(\tau,x) 
	& =
	\frac{1}{\sqrt{2 \pi}} \int_{ \IR} \exp(i x \xi) \hat{v}(\tau,\xi)d\xi 
	\\ & =
	\frac{1}{\sqrt{2 \pi}} \int_{\IR} \exp(i x \xi) \exp\big(-\tau\psi(\xi)\big) \widehat{\varphi \circ \exp}(\xi) d\xi
	\;.
	\end{aligned}
	\]
	For every $k\in \IN_0$,
	Parseval's equality implies with the lower bound in \eqref{eq:StrEllPsi}
	$$
	\begin{array}{rcl}
	\displaystyle
	\int_{\IR} |(D^k_x v)(\tau,x)|^2 dx  
	& = &  \displaystyle 
	\int_{\IR} |\xi|^{2k} \left|\exp\big(-2\tau\psi(\xi)\big)\right| |\widehat{\varphi \circ \exp}(\xi)|^2 d\xi
	\\
	& \leq &  \displaystyle
	\int_{\IR} |\xi|^{2k} \exp(-2\tau C_1 |\xi|^{2\rho}) |\widehat{\varphi \circ \exp}(\xi)|^2 d\xi
	\;.
	\end{array}
	$$
	An elementary calculation shows that 
	for any $m,\kappa,\mu>0$ holds 
	\begin{equation}\label{eq:maxexp}
	\max_{\eta>0} \left\{ \eta^m \exp(-\kappa \eta^\mu) \right\} 
	=
	\left(\frac{m}{\kappa \mu e} \right)^{m/\mu} \;.
	\end{equation}
	We employ \eqref{eq:maxexp} with 
	$m=2k$, $\kappa=2\tau C_1$, $\mu = 2\rho$ and $\eta = |\xi|$ 
	to obtain
	\begin{equation*}
	\| (D^k_xv)(\tau,\cdot) \|_{L^2(\IR)}^2
	\leq 
	\left(\frac{k}{2\tau C_1 \rho e}\right)^{k/\rho} 
	\| v_0 \|_{L^2(\IR)}^2\;.
	\end{equation*}
	Taking square roots and using the (rough) Stirling bound
	$k^k \leq k!e^k$ valid for all $k\in \IN$,
	we obtain %
	\begin{equation}\label{eq:GevEst}
	\forall \tau>0, \forall k\in \IN: \quad
	\| (D^k_xv)(\tau,\cdot) \|_{L^2(\IR)}
	\leq 
	\Bigg( \left(\frac{1}{2\tau C_1\rho}\right)^{\frac{1}{2\rho}} \Bigg)^k (k!)^{\frac{1}{2\rho}} \| v_0 \|_{L^2(\IR)}\;.
	\end{equation}
	This implies, with the Sobolev embedding theorem, that
	for any bounded interval $I = [x_-,x_+]\subset \IR$, $-\infty < x_- < x_+ < \infty$,
	and for every fixed $\tau>0$,
	there exist constants $C= C(x_+, x_-)>0$ and $A(\tau,\rho)>0$ such that
	\begin{equation*}
	\forall k\in \IN: \quad 
	\sup_{x \in I} |  (D^k_xv)(\tau,x) |
	\leq 
	C \big(A(\tau,\rho)\big)^k (k!)^{1/\min\{1, 2\rho\}} \;.
	\end{equation*}
	I.e., 
	$v(\tau,\cdot)|_I$ is Gevrey-$\delta$ regular with $\delta = 1/\min\{1, 2\rho\}$.
	
	To construct the DNNs $\psi^u_\varepsilon$ in the claim, 
	we proceed in several steps: 
	we first use a (analytic, in the bounded interval 
	$I = [x_-,x_+]\subset \IR$) 
	change of variables $s = \exp(x)$ and the 
	fact that Gevrey regularity
	is preserved under analytic changes of variables
	to infer Gevrey-$\delta$ regularity in 
	$[a,b]\subset \IR_{>0}$ of $s\mapsto u(\tau,s)$,
	for every fixed $\tau>0$.
	
	This, in turn, implies the existence of a sequence $\{u_p(s)\}_{p\geq 1}$ 
	of polynomials of degree $p\in \IN$ in $[a,b]$ converging 
	in $W^{1,\infty}([a,b])$ to $u(\tau,\cdot)$ for $\tau>0$
	at rate $\exp(-b'p^{1/\delta})$ 
	for some constant $b'> 0$ depending on $a$, $b$ and on $\delta \geq 1$, 
	but independent of $p$.
	The asserted DNNs are then obtained by 
	approximately expressing the $u_p$ through
	ReLU DNNs, again at exponential rates, 
	with Opschoor et al.\ \cite{OSZ19_839}.
	The details are as follows.
	
	The interval $s\in [a,b]$ in the assertion of the proposition
	corresponds to $x\in [\log(a),\log(b)]$ under the 
	analytic (in the bounded interval $[a,b]$) change of variables $x=\log(s)$. 
	As Gevrey regularity is known to be 
	preserved under analytic changes of variables
	(e.g.\ \cite[Proposition 1.4.6]{Rodino}), 
	also $u(\tau,s)|_{s \in [a,b]}$ is Gevrey-$\delta$ regular, with
	the same index $\delta = 1/ \min\{1, 2\rho\} \geq 1$ and with 
	constants in the derivative bounds which depend on 
	$0 < a < b < \infty$, $\rho \in (0,1]$, $\tau > 0$.
	In particular, 
	for $\rho \geq 1/2$, $u(\tau,s)|_{s \in [a,b]}$ is 
	real analytic in $[a,b]$.
	
	With Gevrey-$\delta$ regularity of $s\mapsto u(\tau,s)$ for $s \in [a,b]$
	established, we may invoke expression rate bounds for deep ReLU NNs 
	for such functions: 
	in Opschoor et al.\ \cite[Proposition 4.1]{OSZ19_839}, it was shown that for such functions
	in space dimension $d=1$ there exist constants $C'>0,\beta'>0$ such that
	for every $\cN \in \IN$ there exists a deep ReLU NN $\tilde{u}_\cN$ with 
	\begin{gather*}
	M(\tilde{u}_\cN) \leq \cN,
	\qquad
	L(\tilde{u}_\cN) \leq C' \cN^{\min\left\{\tfrac12,\tfrac{1}{d+1/\delta}\right\}}\log(\cN),
	\\
	\normc[{W^{1,\infty}([-1,1]^d)}]{u - \mathrm{R}(\tilde{u}_\cN)}
	\leq
	C' \exp\left(- \beta'\cN^{\min\left\{\tfrac{1}{2\delta},\tfrac{1}{d\delta+1}\right\}} \right).
	\end{gather*}
	This implies that for every $0<\eps<1/2$,
	a pointwise error of $\cO(\eps)$ in $[a,b]$ 
	can be achieved by some ReLU NN $\psi^u_\eps$ 
	of 
	depth $\cO(|\log(\eps)|^{\delta} |\log(|\log(\eps)|)|)$
	and of 
	size $\cO(|\log(\eps)|^{2\delta})$.

	This completes the proof.
\end{proof}
%%%%%%%%%%%%%%%%%%%%%%%%%%%%%%%%%%%%%%%%%%%%%
\subsection{Summary and Discussion}
\label{sec:SumDisc}
%%%%%%%%%%%%%%%%%%%%%%%%%%%%%%%%%%%%%%%%%%%55
For prices of derivative contracts 
on one risky asset, whose log-returns are
modelled by a LP $X$, we have analyzed the 
expression rate of deep ReLU NNs. 
We provided two mathematically distinct approaches
to the analysis of the expressive power of deep
ReLU NNs. 
The first, probabilistic approach
furnished algebraic expression rates, i.e.\
pointwise accuracy $\varepsilon>0$ on
a bounded interval $[a,b]$ was furnished with 
DNNs of size $\cO(\varepsilon^{-q})$ with 
suitable $q\geq 0$. 
The argument is based on approximating the option price by Monte Carlo sampling, 
estimating the uniform error on $[a,b]$ and then emulating the resulting average by a DNN.
The second, ``analytic'' approach, leveraged
regularity of (variational) solutions of the corresponding
Kolmogorov partial integrodifferential equations,
and furnished exponential rates of DNN expression.
That is, expression error $\varepsilon > 0$ is achieved
with DNNs of size $\cO(|\log(\varepsilon)|^a)$ for suitable $a>0$.
Key in the second approach were stronger
conditions \eqref{eq:StrEllPsi} on the characteristic 
exponent of the LP $X$, which imply, as we showed, 
Gevrey-$\delta$ regularity of the map $s\mapsto u(\tau,s)$
for suitable $\tau>0$. This regularity implies, in turn,
exponential rates of polynomial approximation (in the uniform
norm on $[a,b]$) of $s\mapsto u(\tau,s)$, 
which is a result of independent interest and, subsequently, 
by emulation of polynomials with deep ReLU NNs, 
the corresponding exponential rates.

We remark that in the particular case $\delta = 1$, 
the derivative bounds \eqref{eq:GevEst} 
imply analyticity of the map $s\mapsto u(\tau,s)$ for $s\in [a,b]$ 
which implies the assertion also with the exponential expression rate 
bound for analytic functions in Opschoor et al.\ \cite{OSZ19_839}.

We also remark that the smoothing of the solution operator
in Proposition \ref{prop:Holom} accommodated payoff functions
which belong merely to $L^2$, 
as arise e.g.\ in particular binary contracts. 
This is a consequence of the assumption \eqref{eq:StrEllPsi} 
which, on the other hand, excludes L\'evy processes with one-sided jumps. 
Such processes are covered by Proposition \ref{prop:1dresult}.
%%%%%%%%%%%%%%%%%%%%%%%%%%%%%%%%%%%%%%%%%%%%%%%%%%%%%%%%%%%%%%%%%
\section{DNN approximation rates for multivariate L\'evy models}
\label{sec:DNNMultivLevy}
%%%%%%%%%%%%%%%%%%%%%%%%%%%%%%%%%%%%%%%%%%%%%%%%%%%%%%%%%%%%%%%%%%
We now turn to DNN expression rates for multivariate geometric L\'evy models.
This is a typical situation when option prices on baskets of $d$ risky assets
are of interest, whose log-returns are modelled by multivariate L\'evy processes.
We admit rather general jump measures with, in particular, fully correlated 
jumps in the marginals, as provided, for example, 
by so-called L\'evy copula constructions in Kallsen and Tankov \cite{KT06}.

As in the univariate case, 
we prove two results on ReLU DNN expression rates 
of option prices for European style contracts. 
The first argument is developed in Section \ref{sec:DNNProbab}
below and overcomes, in particular, the curse of dimensionality.
Its proof is again based on 
probabilistic arguments from statistical learning theory.
As exponential LPs $X^d$ generalize geometric Brownian motions,
Theorem \ref{prop:Ddresult} generalizes several results 
from the classical Black--Scholes setting and
we comment on the relation of Theorem \ref{prop:Ddresult} 
to these recent results in Section \ref{sec:RelatedLit}.
Owing to the method of proof, the DNN expression rate in Theorem
\ref{prop:Ddresult} will deliver an $\varepsilon$-complexity
of $\cO(\varepsilon^{-2})$, achieved with potentially shallow DNNs, 
see Remark~\ref{rem:architecture}.

The second argument is based on parabolic regularity 
of the deterministic Kolmogorov PIDE associated to the LP $X^d$.
We show in Theorem \ref{thm:DNNExprd} that polylogarithmic in 
$\varepsilon$ expression rate bounds can be achieved
by allowing DNN depth to increase essentially as $\cO(|\log \varepsilon|)$.
The result in Theorem \ref{thm:DNNExprd} is, however, prone to the 
curse of dimensionality: constants implied in the $\cO(\cdot)$ bounds 
may (and, generally, will) depend exponentially on $d$.
We also show that
under a hypothesis on sufficiently large time $t>0$, 
parabolic smoothing will allow to overcome the curse of dimension,
with dimension-independent expression rate bounds which are possibly
larger than the rates furnished by the probabilistic argument (which is,
however, valid uniformly for all $t>0$).
%%%%%%%%%%%%%%%%%%%%%%%%%%%%%%%%%%%%%%%%%%%%%%%%%%%%%%%%%%%%%%%%%
\subsection{DNN expression rate bounds via probabilistic argument}
\label{sec:DNNProbab}
%%%%%%%%%%%%%%%%%%%%%%%%%%%%%%%%%%%%%%%%%%%%%%%%%%%%%%%%%%%%%%%%%
We start by remarking that in this subsection,
there is no need to assume ReLU activation.

The following result proves that neural networks are capable of approximating 
option prices in multivariate exponential L\'evy models 
without the curse of dimensionality given that the corresponding 
L\'evy triplets $(A^d,\gamma^d,\nu^d)$
are bounded uniformly with respect to the dimension $d$. 

For any dimension $d \in \N$ we assume given a payoff 
$\varphi_d \colon \R^d \to \R$, 
a $d$-variate LP $X^d$ and 
we denote the option price in time-to-maturity 
by
\begin{equation} \label{eq:optionPriceD}
u_d(\tau,s) 
= 
\E\left[\varphi_d\big(s_1 \exp(X_{\tau,1}^d),\ldots,s_d \exp(X_{\tau,d}^d)\big)\right], 
\quad \tau \in [0,T], s \in (0,\infty)^d.
\end{equation}
We refer to Sato \cite{Sato1999} for more details on multivariate L\'evy processes 
and to Cont and Tankov \cite{Cont2004}, Eberlein and Kallsen \cite{EberKall19} for more details on multivariate geometric L\'evy models in finance.

The next theorem is a main result of the present paper.
It states that DNNs can efficiently 
express prices on possibly large baskets of risky assets 
whose dynamics are driven by multivariate L\'evy processes
with general jump correlation structure.
The expression rate bounds are polynomial in the number $d$
of assets and, therefore, not prone to the curse of dimensionality.
This result partially generalizes earlier work on DNN expression rates for diffusion 
models in Elbr\"achter et al.\ \cite{EGJS18_787}, Grohs et al.\ \cite{HornungJentzen2018}.
\begin{theorem} \label{prop:Ddresult}
	Assume that for any $d \in \N$,
	the payoff $\varphi_d \colon \R^d \to \R$ can be 
	approximated well by neural networks, that is, 
	there exists constants $c>0$, $p \geq 2,\tilde{q},q\geq 0$ and, 
	for all $\varepsilon \in (0,1]$, $d \in \N$,
	there exists a neural network $\phi_{\varepsilon,d}$ 
	with  
	\begin{align} \label{eq:NNcloseD}
	|\varphi_d(s)-\Rl(\phi_{\varepsilon,d})(s)| 
	& \leq \varepsilon c d^{\tilde{q}} (1+\|s\|^p), \quad \text{ for all } s \in (0,\infty)^d,
	\\ \label{eq:NNsparseD} 
	M(\phi_{\varepsilon,d}) &\leq c d^{\tilde{q}} \varepsilon^{-q}, 
	\\ \label{eq:NNlipschitzD} 
	\mathrm{Lip}\big(\Rl(\phi_{\varepsilon,d})\big) & \leq c d^{\tilde{q}}.
	\end{align}
	In addition, assume that the L\'evy triplets $(A^d,\gamma^d,\nu ^d)$ of $X^d$ are 
	bounded in the dimension, that is, 
	there exists a constant $B > 0$ such that 
	for each $d \in \N$, $i,j=1,\ldots,d$,
	\begin{equation}
	\label{eq:LevyTripletMultiD}
	\max\left\{A^d_{ij},
	\gamma^d_i ,
	\int_{\{\|y\|> 1\}} e^{p y_i} \nu^d(d y),
	\int_{\{\|y\|\leq 1\}} y_i^2 \nu^d(d y)\right\}  \leq B. 
	\end{equation}
	Then there exist constants $\kappa,\mathfrak{p},\mathfrak{q} \in [0,\infty)$ and 
	neural networks $\psi_{\varepsilon,d}$, $\varepsilon \in (0,1]$, $d \in \N$  
	such that for any target accuracy $\varepsilon \in (0,1]$ and for any  $d \in \N$  
	the number of weights grows only polynomially 
	$M(\psi_{\varepsilon,d}) \leq \kappa d^{\mathfrak{p}}\varepsilon^{-\mathfrak{q}}$ 
	and the approximation error between the neural network $\psi_{\varepsilon,d}$ 
	and the option price is at most $\varepsilon$, that is, 
	\[
	\sup_{s \in [a,b]^d} |u_d(T,s) - \Rl(\psi_{\varepsilon,d})(s)| \leq \varepsilon.
	\]
\end{theorem}

\begin{remark}\label{rmk:logGrow}
	The statement of Theorem~\ref{prop:Ddresult} is still valid, 
	if we admit logarithmic growth of $B$ with $d$ in \eqref{eq:LevyTripletMultiD}. 
\end{remark}

\begin{remark}
	As in the univariate case (cf.\ Remark~\ref{rmk:q=0}), 
	in relevant examples of options written on $d>1$ underlyings 
	(such as basket options, call on max/min options, put on max/min options, \ldots)
	the payoff can be represented \textit{exactly} 
	as a ReLU DNN. Thus,
	we may choose $q=0$ in \eqref{eq:NNsparseD} and obtain 
	$\mathfrak{q} = 2$ in Theorem~\ref{prop:Ddresult} (cf.\ \eqref{eq:auxEq25}).
\end{remark}
\begin{proof}
	Let $\varepsilon \in (0,1]$ be the given target accuracy and 
	consider $\bar{\varepsilon} \in (0,1]$ (to be selected later). 
	To simplify notation we write for $s \in [a,b]^d$ 
	\[
	s e^{X_{T}^d} = \left(s_1 \exp(X_{T,1}^d),\ldots,s_d \exp(X_{T,d}^d)\right).
	\]
	The proof consists in four steps:
	\begin{itemize}
		\item 
		Step 1 bounds the error that arises when the payoff $\varphi_d$ 
		is replaced by the neural network approximation $\phi_{\bar{\varepsilon},d}$. 
		As a part of Step 1 we also prove that the $p$-th exponential moments of the 
		components $X_{T,i}^d$ of the L\'evy process are bounded uniformly in the dimension $d$. 
		\item 
		Step 2 is a technical step that is required for Step 3; 
		it bounds the error that arises when the L\'evy process is capped at a threshold $D>0$.
		If we assumed in addition that the output of the neural network $\phi_{\bar{\varepsilon},d}$ 
		were bounded (this is for example the case if the activation function $\varrho$ is bounded), 
		then Step 2 could be omitted.
		\item 
		Step 3 is the key step in the proof. 
		We introduce $n$ i.i.d.\ copies of (the capped version of) 
		$X_T^d$ and use statistical learning techniques 
		(symmetrization, Gaussian and Rademacher complexities) to estimate the 
		expected maximum difference between the option price (with neural network payoff) 
		and its sample average. 
		This is then used to construct the approximating neural networks.
		\item Step 4 combines the estimates from Steps 1-3 and concludes the proof. 
	\end{itemize}
	
	\textit{Step 1:}
	Assumption \eqref{eq:NNcloseD} and H\"older's inequality yield for all $s \in [a,b]^d$
	\begin{align}\nonumber
	|u_d(T,s) - \E[\Rl(\phi_{\bar{\varepsilon},d})(se^{X_T^d})]| 
	& \leq \E[|\varphi_d(se^{X_T^d}) - \Rl(\phi_{\bar{\varepsilon},d})(se^{X_T^d})|] 
	\\ \nonumber 
	&  \leq \bar{\varepsilon} c d^{\tilde{q}} (1+\E[\|se^{ X_T^d}\|^p])
	\\ \nonumber 
	&  = \bar{\varepsilon} c d^{\tilde{q}} \left(1+\E\left[\left(\sum_{i=1}^{d} s_i^2 e^{2 X_{T,i}^d}\right)^{p/2}\right]\right)
	\\ \nonumber 
	&  \leq \bar{\varepsilon} c d^{\tilde{q}} \left(1+b^p \E\left[d^{(p-1)/2}(\sum_{i=1}^{d} e^{2 p X_{T,i}^d})^{1/2}\right]\right)
	\\ &  
	\leq \bar{\varepsilon} c_1 d^{{\tilde{q}}+\frac{1}{2}p+\frac{1}{2}} 
	\label{eq:auxEq7} \end{align}
	with the constant $c_1 = c\max\{1,b^p\}  (1+ \sup_{d,i} \E[e^{p X_{T,i}^d}])$ 
	and we used that  $\|\cdot\| \leq \|\cdot\|_1$ in the last step.
	To see that $c_1$ is indeed finite, 
	note that \eqref{eq:LevyTripletMultiD} and \cite[Theorem~25.17]{Sato1999} 
	(with the vector $w\in \IR^d$ in that result being $pe_i$) 
	imply that for any $d \in \N$, $i=1,\ldots,d$,
	the exponential moment can be bounded as 
	\begin{align} 
	\nonumber
	\E[e^{p X_{T,i}^d}] & = \exp\left(T\Big(\frac{p^2}{2}A^d_{ii} + \int_{\R^d} (e^{p y_i}-1-p y_i \mathbbm{1}_{\{\|y\|\leq 1\}}) \nu^d(d y) + p \gamma^d_i \Big)\right)
	\\ \nonumber & \leq \exp\Bigg(T\Big(\frac{3 p^2}{2}B + \int_{\{\|y\|\leq 1\}} (e^{p y_i}-1-p y_i) \nu^d(d y) \\ \nonumber & \quad \quad + \int_{\{\|y\|> 1\}} (e^{p y_i}-1) \nu^d(d y) \Big)\Bigg)
	\\ \nonumber & \leq \exp\left(T\Big(\frac{5 p^2}{2}B + p^2 e^{p} \int_{\{\|y\|\leq 1\}} y_i^2 \nu^d(d y) \Big)\right)
	\\ & \leq \exp\left(T\Big(\frac{5 p^2}{2}B + p^2 e^{p} B \Big)\right),
	\label{eq:expMomentMultiD} 
	\end{align}
	where in the second inequality we used that
	$ |e^{z}-1-z| \leq z^2 e^p$ for all $z \in [-p,p]$
	which can be seen e.g.\ from the (mean value form of the) 
	Taylor remainder formula.
	
	\textit{Step 2:}
	Before proceeding with the key step of the proof, we need to introduce a cut-off in order to ensure that the neural network output is bounded. Let $D>0$ and consider the random variable $X_T^{d,D} = \min(X_T^d,D)$, where the minimum is understood componentwise. Then the Lipschitz property \eqref{eq:NNlipschitzD} implies
	
	\begin{align} \nonumber
	|\E[\Rl(\phi_{\bar{\varepsilon},d})(se^{X_T^d})] &  - \E[\Rl(\phi_{\bar{\varepsilon},d})(se^{X_T^{d,D}})]| \\ \nonumber & \leq c d^{\tilde{q}} \E[\|se^{X_T^d} - se^{X_T^{d,D}}\|] 
	\\ \nonumber &  \leq  b c d^{\tilde{q}} \E\left[\sum_{i=1}^{d} |e^{X_{T,i}^d} - e^{X_{T,i}^{d,D}}|\right]
	\\\nonumber &  \leq  b c d^{\tilde{q}} \sum_{i=1}^{d} \E\left[ 2e^{X_{T,i}^d} \mathbbm{1}_{\{X_{T,i}^d>D\}}\right]
	\\\nonumber &  \leq  2b c d^{\tilde{q}} \sum_{i=1}^{d} \E[e^{2X_{T,i}^d}]^{1/2} \P[X_{T,i}^d>D]^{1/2}
	\\\nonumber &  \leq  2 e^{-D} b c d^{\tilde{q}} \sum_{i=1}^{d} \E[e^{2X_{T,i}^d}]
	\\ &  \leq \tilde{c}_1  e^{-D} d^{{\tilde{q}}+1} ,
	\label{eq:auxEq12}  \end{align}
	where $\tilde{c}_1 = 2 b c \exp(5TpB + 2Te^{p}pB)$ and we used $\|\cdot\| \leq \|\cdot\|_1$, H\"older's inequality, Chernoff's bound and finally again H\"older's inequality and \eqref{eq:expMomentMultiD}.
	
	\textit{Step 3:}
	Let $X_1,\ldots,X_n$ denote $n$ i.i.d.\ copies of the random vector $X_{T}^{d,D}$ and
	let $Z_1,\ldots,Z_n$ denote i.i.d.\ standard normal variables, independent of $X_1,\ldots,X_n$. For any separable class of functions $\Hc \subset C(\R^d,\R)$ define the random variable (the so-called empirical Gaussian complexity)
	\[
	\hat{G}_n(\Hc) = \E\left[ \left. \sup_{f \in \Hc} \Big|  \frac{2}{n} \sum_{k=1}^n Z_k f(X_k) \Big| \right| X_1,\ldots,X_n \right].
	\] 
	Consider now for $i=1,\ldots,d$ the function classes 
	\[ \Hc_i = \{(-\infty,D]^d \ni x \mapsto s \exp(x_i) \colon s \in [a,b]
	\}
	\]
	and, with the notation $s \exp(x)=(s_1 \exp(x_1),\ldots,s_d\exp(x_d))$, 
	the class
	\[ 
	\Hc = \{ (-\infty,D]^d \ni x \mapsto 
	\Rl(\phi_{\bar{\varepsilon},d})\big(s \exp(x)\big)-\Rl(\phi_{\bar{\varepsilon},d})(0) \colon s \in [a,b]^d  \}. 
	\]
	Denoting by $\tilde{\Hc} \subset C((-\infty,D]^d,\R^d)$ the direct sum of $\Hc_1,\ldots,\Hc_d$, we have that
	\[
	\Hc = \phi(\tilde{\Hc})
	\]
	where $\phi = \Rl(\phi_{\bar{\varepsilon},d})(\cdot)-\Rl(\phi_{\bar{\varepsilon},d})(0)$ 
	is a Lipschitz-function with Lipschitz-constant $cd^{\tilde{q}}$ 
	(due to hypothesis on the Lipschitz-constant of the neural network \eqref{eq:NNlipschitzD}) and $\phi$ satisfies
	$\phi(0)=0$ and $\phi$ is bounded on the range of $\tilde{\Hc}$ 
	(which is contained in $[0,b\exp(D)]^d$).
	
	Consequently, Theorem~14 in Bartlett and Mendelson \cite{Bartlett2003} implies that 
	\begin{equation}\label{eq:auxEq11}
	\hat{G}_n(\Hc) \leq 2 c d^{\tilde{q}} \sum_{i=1}^d \hat{G}_n(\Hc_i).
	\end{equation}
	Therefore, denoting by $\varepsilon_1,\ldots,\varepsilon_n$ 
	an independent collection of Rademacher random variables, 
	we estimate
	\begin{align} \nonumber
	& \E\left[\sup_{s \in [a,b]^d} \left| \E[\Rl(\phi_{\bar{\varepsilon},d})(se^{X_T^{d,D}})] - \frac{1}{n} \sum_{k=1}^n \Rl(\phi_{\bar{\varepsilon},d})(se^{X_k}) \right| \right]  
	\\ \nonumber& \leq 2
	\E\left[\sup_{s \in [a,b]^d} \left| \frac{1}{n} \sum_{k=1}^n \varepsilon_k \phi(se^{X_k}) \right| \right]  
	\\ \nonumber& \leq \tilde{c_2}
	\E\left[\sup_{s \in [a,b]^d} \left| \frac{2}{n} \sum_{k=1}^n Z_k \phi(se^{X_k}) \right| \right] 
	\\ \nonumber& = \tilde{c_2}
	\E\left[\sup_{f \in \Hc} \left| \frac{2}{n} \sum_{k=1}^n Z_k f(X_k) \right| \right] 
	\\ \nonumber& \leq 2 \tilde{c_2} c d^{\tilde{q}} \sum_{i=1}^d \E[\hat{G}_n(\Hc_i)]
	\\ & \leq \frac{4 \tilde{c_2} c d^{\tilde{q}} b }{n} \sum_{i=1}^d \E\left[\left| \sum_{k=1}^n Z_k e^{X_{k,i}} \right| \right].
	\label{eq:auxEq10} \end{align}
	Here, 
	the first inequality follows by symmetrization (see for example Boucheron et al.\ \cite[Lemma~11.4]{Boucheron2013}), 
	the second inequality follows from the comparison results on Gaussian and Rademacher complexities 
	(see for instance Bartlett and Mendelson \cite[Lemma~4]{Bartlett2003}) 
	with some absolute constant $\tilde{c_2}$ and 
	the third inequality uses \eqref{eq:auxEq11}.
	
	In fact, it is possible to prove that the constant  $\tilde{c_2}$ in \eqref{eq:auxEq10} may be chosen as $\tilde{c_2} =  1/\E[|Z_1|] = \sqrt{\pi/2}$. 
	Indeed, setting $\Gc=\sigma(\varepsilon_1,\ldots,\varepsilon_n,X_1,\ldots,X_n)$ and using independence yields
	
	\[\begin{aligned}
	\E[|Z_1|] 
	& \E\left[\sup_{s \in [a,b]^d} \left| \frac{1}{n} \sum_{k=1}^n \varepsilon_k \phi(se^{X_k}) \right| \right] 
	\\ & 
	= \E\left[\sup_{s \in [a,b]^d} \left| \frac{1}{n} \sum_{k=1}^n \E[|Z_k| | \Gc ]\varepsilon_k \phi(se^{X_k}) \right| \right] 
	\\ & 
	= \E\left[\sup_{s \in [a,b]^d} \left|  \E\left[\left.\frac{1}{n} \sum_{k=1}^n |Z_k| \varepsilon_k \phi(se^{X_k}) \right|  \Gc \right] \right| \right] 
	\\ & \leq 
	\E\left[\E\left[\sup_{s \in [a,b]^d} \left. \big| \frac{1}{n} \sum_{k=1}^n |Z_k| \varepsilon_k \phi(se^{X_k}) \big| \right|  \Gc \right] \right] 
	\\ & =
	\E\left[\sup_{s \in [a,b]^d} \left| \frac{1}{n} \sum_{k=1}^n Z_k \phi(se^{X_k}) \right|   \right] .
	\end{aligned}
	\]
	
	To further simplify \eqref{eq:auxEq10},
	we now apply Jensen's inequality and use independence 
	and $\E[Z_k Z_l] = \delta_{k,l}$ 
	to derive for $i=1,\ldots,d$
	\[ 
	\begin{aligned} 
	\E\left[\left| \sum_{k=1}^n Z_k e^{X_{k,i}} \right| \right] 
	& \leq  \E\left[\left| \sum_{k=1}^n Z_k e^{X_{k,i}} \right|^2 \right]^{1/2} 
	=  \left(\sum_{k=1}^n \E[e^{2 X_{k,i}} ]\right)^{1/2} \\ 
	& \leq \sqrt{n} \E[e^{2 X_{T,i}^d}]^{1/2}  
	\leq \sqrt{n} \E[e^{p X_{T,i}^d}]^{1/p}.
	\end{aligned} 
	\]
	
	Combining this with the previous estimate \eqref{eq:auxEq10} and 
	with the exponential moment estimate \eqref{eq:expMomentMultiD} 
	we obtain that 
	\begin{equation*}
	\E\left[\sup_{s \in [a,b]^d} 
	\left| \E[\Rl(\phi_{\bar{\varepsilon},d})(se^{X_T^{d,D}})] - \frac{1}{n} \sum_{k=1}^n \Rl(\phi_{\bar{\varepsilon},d})(se^{X_k}) \right| \right]  
	\leq 
	\frac{c_2 d^{{\tilde{q}}+1}}{\sqrt{n}} 
	\end{equation*}
	with $c_2 = 4 \sqrt{\pi/2} c b \exp\left(5BT p/2 + BTp e^{p}\right)$. 
	By applying Markov's inequality (see \eqref{eq:auxEq14}-\eqref{eq:auxEq16}) 
	this proves that there exists $\omega \in \Omega$ with 
	\begin{equation*}
	\sup_{s \in [a,b]^d} 
	\left| 
	\E[\Rl(\phi_{\bar{\varepsilon},d})(se^{X_T^{d,D}})] - \frac{1}{n} \sum_{k=1}^n \Rl(\phi_{\bar{\varepsilon},d})(se^{X_k(\omega)}) 
	\right|
	\leq 
	\frac{2 c_2 d^{{\tilde{q}}+1}}{\sqrt{n}}.
	\end{equation*}
	Now, we observe 
	that $s \mapsto \frac{1}{n} \sum_{k=1}^n \Rl(\phi_{\bar{\varepsilon},d})(se^{X_k(\omega)})$ 
	is the realization of a neural network 
	$\tilde{\psi}_{\bar{\varepsilon},d}$ with $M(\tilde{\psi}_{\bar{\varepsilon},d}) \leq n M(\phi_{\bar{\varepsilon},d})$ 
	(see Lemma~\ref{lem:averageNN}).
	We have therefore proved that for arbitrary $n\in \IN$ 
	there exists a neural network $\tilde{\psi}_{\bar{\varepsilon},d}$ 
	with
	\begin{equation}\label{eq:auxEq18}
	\sup_{s \in [a,b]^d} 
	\left| \E[\Rl(\phi_{\bar{\varepsilon},d})(se^{X_T^{d,D}})] -  \Rl(\tilde{\psi}_{\bar{\varepsilon},d})(s) \right| 
	\leq 
	\frac{2 c_2 d^{{\tilde{q}}+1}}{\sqrt{n}}.
	\end{equation}
	
	\textit{Step 4:} 
	In the final step we now provide appropriate choices of the hyperparameters. 
	We select $\bar{\varepsilon} = \varepsilon (c_1 d^{{\tilde{q}}+\frac{1}{2}p+\frac{1}{2}}+2)^{-1}$, 
	choose $n = \lceil(2 c_2 d^{{\tilde{q}}+1} \bar{\varepsilon}^{-1})^2 \rceil$, 
	$D= \log(\bar{\varepsilon}^{-1}d^{{\tilde{q}}+1} \tilde{c}_1)$ 
	and set 
	$\psi_{\varepsilon,d} = \tilde{\psi}_{\bar{\varepsilon},d}$. 
	Then the total number of parameters of the approximating neural network 
	can be estimated using assumption \eqref{eq:NNsparseD} 
	as
	\begin{align} \nonumber
	M(\psi_{\varepsilon,d}) 
	= M(\tilde{\psi}_{\bar{\varepsilon},d}) 
	& \leq n M(\phi_{\bar{\varepsilon},d}) 
	\\  \nonumber
	&  \leq (1+(2 c_2 d^{{\tilde{q}}+1} \bar{\varepsilon}^{-1})^2) c d^{\tilde{q}} \bar{\varepsilon}^{-q}
	\\  \nonumber
	&  \leq (1+4 c_2^2) c d^{3{\tilde{q}}+2} \bar{\varepsilon}^{-2-q}
	\\ 
	& \leq \big((1+4 c_2^2) c(c_1 +2)^{2+q}\big) d^{({\tilde{q}}+\frac{1}{2}p+\frac{1}{2})(2+q)+3{\tilde{q}}+2} \varepsilon^{-2-q},
	\label{eq:auxEq25} \end{align}
	which shows the number of weights to be bounded polynomially in $d$ 
	and $\varepsilon^{-1}$, as claimed. 
	
	Finally, we combine \eqref{eq:auxEq7}, \eqref{eq:auxEq12} and \eqref{eq:auxEq18} 
	to estimate the approximation error as
	\[\begin{aligned}
	\sup_{s \in [a,b]^d} & |u_d(T,s) - \Rl(\psi_{\varepsilon,d})(s)| 
	\\ 
	& \leq \sup_{s \in [a,b]^d} 
	\left( |u_d(T,s) -  \E[\Rl(\phi_{\bar{\varepsilon},d})(se^{X_T^d})]| |\right. \\ & \quad \quad  + \left. 
	| \E[\Rl(\phi_{\bar{\varepsilon},d})(se^{X_T^d})]-\E[\Rl(\phi_{\bar{\varepsilon},d})(se^{X_T^{d,D}})] |\right. 
	\\ & \quad \quad  + \left. |\E[\Rl(\phi_{\bar{\varepsilon},d})(se^{X_T^{d,D}})]- \Rl(\tilde{\psi}_{\bar{\varepsilon},d})(s)|\right)
	\\ & \leq \bar{\varepsilon} c_1 d^{{\tilde{q}}+\frac{1}{2}p+\frac{1}{2}} + \tilde{c}_1  e^{-D} d^{{\tilde{q}}+1} + \frac{2 c_2 d^{{\tilde{q}}+1}}{\sqrt{n}}
	\\ & \leq \bar{\varepsilon} (c_1 d^{{\tilde{q}}+\frac{1}{2}p+\frac{1}{2}}+2) = \varepsilon,
	\end{aligned}\]
	as claimed.
\end{proof}

%%%%%%%%%%%%%%%%%%%%%%%%%%%%%%%%%%%%%%%%%%%%%%%%%%%%%%
\subsection{Discussion of related results}
\label{sec:RelatedLit}
%%%%%%%%%%%%%%%%%%%%%%%%%%%%%%%%%%%%%%%%%%%%%%%%%%%%%
As recently there have been several results on DNN expression rates 
in high dimensional diffusion models,
a discussion on the relation of the multivariate DNN expression rate result,
Thm.\ref{prop:Ddresult}, to other recent mathematical 
results on DNN expression rate bounds is in order.
Given that geometric diffusion models are particular cases of the presently
considered models (corresponding to $\nu^d = 0$ in the L\'evy triplet),
it is of interest to consider to which extent the DNN expression error bound
Thm.\ref{prop:Ddresult} relates to these results.

Firstly, we 
note that with the exception of 
Gonon et al.\ \cite{Gonon2019} and Elbr\"achter et al.\ \cite{EGJS18_787}, 
previous results in the literature which are concerned with DNN approximation rates 
for Kolmogorov equations for diffusion processes
(see, e.g., Gonon et al.\ \cite{Gonon2019}, Grohs et al.\ \cite{GrohsHornungJentzen2019}, Berner et al.\ \cite{BernerGrohsJentzen2018}, Elbr\"achter et al.\ \cite{EGJS18_787}, Grohs et al.\ \cite{HornungJentzen2018}, Reisinger and Zhang \cite{ReisingerZhang2019} 
and the references therein) 
study approximation with respect to the $L^p$-norm ($p<\infty$), 
whereas in Thm.\ref{prop:Ddresult} we study approximation with respect to the  
$L^\infty$-norm, which requires entirely different techniques. 
While the results in \cite{EGJS18_787} rely on specific structure of the payoff, 
the proof of the expression rates in \cite{Gonon2019} has some similarities with 
the proof of Thm.\ref{prop:Ddresult}. 
However, the novelty in the proof of Thm.\ref{prop:Ddresult} is the use of statistical 
learning techniques (symmetrization, Gaussian and Rademacher complexities) 
which allow for weaker assumptions on the activation function than in \cite{Gonon2019}. 
In addition, the class of PDEs considered in \cite{Gonon2019} (heat equation and related) 
is different than the one considered in Thm.\ref{prop:Ddresult} 
(Black--Scholes PDE and L\'evy PIDE).

Secondly, 
Thm.\ref{prop:Ddresult} is the first result on ReLU DNN expression rates for 
option prices in models with jumps or, equivalently, 
for \emph{partial-integrodifferential equations} 
in non-divergence form
\begin{equation}
\label{eq:MultdPDEx}
\begin{array}{rl}
\partial_t v_d(\tau,x) 
& =  \frac{1}{2}\mathrm{Trace}(A^d D^2_x v_d(\tau,x)) + D_x v_d(\tau,x)\gamma^d 
\\ 
& \quad  + \int_{\R^d } \left(v_d(\tau,x+y)-v_d(\tau,x)-D_x v_d(\tau,x) y \mathbbm{1}_{\{\|y\|\leq 1\}} \right) \nu^d(d y) , 
\\
v_d(0,x) &= (\varphi_d \circ \exp)(x)
\end{array}
\end{equation}
for $x \in \R^d, \tau > 0$ or, when transformed from log-price variables $x_i$ 
to actual price variables $s_i$ via $(s_1,\ldots,s_d)=(\exp(x_1),\ldots,\exp(x_d))$ (and with the convention $s e^y = (s_1e^{y_1},\ldots,s_d e^{y_d})$)
\begin{equation}
\label{eq:MultdPDEs}
\begin{array}{rl}
\partial_t u_d(\tau,s) 
& =  \frac{1}{2} \sum_{i,j=1}^d A^d_{i,j} s_i s_j \partial_{s_i} \partial_{s_j} u_d(\tau,s) 
+ \sum_{i=1}^d s_i \tilde{\gamma}^d_i \partial_{s_i} u_d(\tau,s) 
\\ 
& \;  + \int_{\R^d } 
\left(u_d(\tau,s e^y)-u_d(\tau,s)-\sum_{i=1}^d s_i (e^{y_i}-1) \partial_{s_i} u_d(\tau,s) \right) 
\nu^d(d y) , 
\\
u_d(0,s) &= \varphi_d(s)
\end{array}
\end{equation}
for $s \in (0,\infty)^d, \tau > 0$ 
and with $ \tilde{\gamma}_i^d = \gamma_i^d + \frac{A_{i,i}^d}{2} + \int_{\R^d} (e^{y_i}-1- y_i \mathbbm{1}_{\{\|y\|\leq 1\}}) \nu^d(d y)$  (see for instance \cite[Theorem~4.1]{Hilber2009}).
As in our assumptions also $A^d = 0$ is admissible under suitable conditions on $\nu^d$,
the present ReLU DNN expression rates are not mere generalizations of the diffusion 
case, but cover indeed the case of pure jump models both for
finite and for infinite activity L\'evy processes. 

In the case of $X$ being a diffusion with drift, i.e.\ for $\nu^d=0$, 
the L\'evy PIDE reduces to a Black--Scholes PDE. 
In this particular case, 
we may compare the result in Thm.\ref{prop:Ddresult} 
to the recent results e.g.\ in  Grohs et al. \cite{HornungJentzen2018}. 
The results in the latter article are specialized to the 
Black--Scholes case in Section~4  \cite{HornungJentzen2018}, 
where Setting~4.1 specifies the coefficients $(A^d)_{i,j}$ 
(in our notation) as  $ \beta_i^d \beta_j^d (B^d (B^d)^\top)_{i,j}$ for some 
$\beta^d \in \R^d, B^d \in \R^{d\times d}$ satisfying  $(B^d (B^d)^\top)_{k,k} = 1$ 
for all $d \in \N$, $i,j,k=1,\ldots,d$ and 
$\sup_{d,i} |\beta_i^d| < \infty$. 
The coefficient $\gamma^d$ is chosen as $\alpha^d$ 
satisfying $\sup_{d,i} |\alpha^d_i| < \infty$. 
Using that $\Sigma = (B^d (B^d)^\top)$ is symmetric, positive definite we obtain 
$\Sigma_{i,j} \leq \sqrt{\Sigma_{i,i}\Sigma_{j,j}} = 1$ 
and hence these assumptions imply that \eqref{eq:LevyTripletMultiD} is satisfied. 
Therefore, the DNN expression rate results from Section~4 in \cite{HornungJentzen2018} 
can also be deduced from Thm.\ref{prop:Ddresult},
here in the case when the probability measure used to quantify
the $L^p$-error in \cite{HornungJentzen2018} is compactly supported, 
as in that case the $L^\infty$-bounds proved here imply the 
$L^p$-bounds proved in \cite{HornungJentzen2018}.
%%%%%%%%%%%%%%%%%%%%%%%%%%%%%%%%%%%%%%%
\subsection{Exponential ReLU DNN expression rates via PIDE}
\label{sec:LevyPIDE}
%%%%%%%%%%%%%%%%%%%%%%%%%%%%%%%%%%%%%%%
We now extend the univariate case discussed
in Section \ref{sec:exp}, 
and prove an exponential expression rate bound similar to 
Proposition \ref{prop:Holom} for baskets of $d\geq 2$ 
L\'evy-driven assets. 
In this subsection we assume ReLU activation function $\varrho(x)=\max\{x,0\}$. 
As in Section \ref{sec:DNNProbab}, we admit
general correlation structure of the marginal processes' jumps.
To prove DNN expression rate bounds, 
we exploit once more
the fact that the stationarity and homogeneity 
of the $\R^d$-valued LP $X^d$ imply that 
the Kolmogorov equation \eqref{eq:MultdPDEx}
has constant coefficients.
Under the provision that in \eqref{eq:MultdPDEx} 
holds $v_d(0,\cdot)\in L^2(\IR^d)$,
this allows to write for every $\tau>0$
the Fourier transform $F_{x\to \xi}v_d(\tau,\cdot) = \hat{v}_d(\tau,\xi)$
as
\begin{equation}\label{eq:hatvdxi}
\hat{v}_d(\tau,\xi) = \exp\big(-\tau\psi(\xi)\big) \hat{v}_d(0,\xi) \;,\quad \xi\in \IR^d\;.
\end{equation}
Here, for $\xi \in \IR^d$ the symbol 
$\psi(\xi) = \exp(-ix^\top\xi) \cA(\partial_x) \exp(i x^\top\xi)$ 
with $\cA(\partial_x)$ denoting the constant coefficient spatial 
integrodifferential operator in \eqref{eq:MultdPDEx} by Courr\`{e}ge's 2nd Theorem 
(see, e.g., Applebaum \cite[Theorem~3.5.5]{Applebaum2009}), 
and \eqref{LK1} becomes 
\begin{equation}\label{LKd}
\IE[\exp(i\xi^\top X_\tau^d)] 
= 
\exp\big(-\tau\psi(\xi)\big)\;,\quad \xi \in \IR^d\;.
\end{equation}
In fact, $\psi$ can be expressed in terms of the characteristic triplet $(A^d,\gamma^d,\nu^d)$
of the LP $X^d$ as
\begin{equation}\label{eq:symbolD}
\psi(\xi) 
= 
\frac{1}{2} \xi^\top A^d \xi - i \xi^\top \gamma^d  - \int_{\R^d } 
\left(e^{i \xi^\top y}-1- i \xi^\top y \mathbbm{1}_{\{\|y\|\leq 1\}} \right) \nu^d(d y)\;,
\quad \xi \in \IR^d\;.  
\end{equation}
We impose again the strong ellipticity assumption \eqref{eq:StrEllPsi},
however now with $|\xi|$ understood as
$|\xi|^2 = \xi^\top \xi$ for $\xi\in \IR^d$.
Then reasoning exactly as in the proof of Proposition \ref{prop:Holom}
we obtain with $C_1>0$ as in \eqref{eq:StrEllPsi} 
for every $\tau>0$ for the variational solution $v_d$ 
of \eqref{eq:MultdPDEx} the bound
\begin{equation}\label{eq:L2Estd}
\forall k\in \IN_0:\quad 
\| (D^k_x v_d)(\tau,\cdot) \|_{L^2(\IR^d)}^2
\leq 
\left(\frac{k}{2\tau C_1 \rho e}\right)^{k/\rho}
\| v_d(0,\cdot) \|_{L^2(\IR^d)}^2
\;.
\end{equation}
Here, $D^k_x$ denotes any weak derivative of total order $k\in \IN_0$ 
with respect to $x\in \IR^d$.

With the Sobolev embedding theorem we again obtain
for any bounded cube $I^d = [x_-,x_+]^d\subset \IR^d$ with $-\infty < x_- < x_+ < \infty$,
and for every fixed $\tau>0$,
that there exist constants $C(d)>0$ and $A(\tau,\rho) > 0$ 
such that
\begin{equation}
\label{eq:utGevrxD}
\forall k\in \IN: \quad
\sup_{x \in I^d}|  (D^k_xv_d)(\tau,x) |
\leq
C(d) \big(A(\tau,\rho)\big)^k (k!)^{1/\min\{1, 2\rho\}} \;.
\end{equation}
The constant $C(d)$ is independent of $x_-,x_+$, 
but depends in general exponentially 
on the basket size (respectively the dimension) $d\geq 2$, 
and the constant $A(\tau,\rho) = (2\tau C_1\rho)^{-1/(2\rho)}$ 
denotes the constant from \eqref{eq:L2Estd} and Stirling's bound.
If $\rho =1$ (which corresponds to the case of non-degenerate diffusion) 
and 
if $\tau>0$ is sufficiently large (so that $(2 \tau C_1)^{1/(2\rho)} \geq 1$) 
then the constant is bounded uniformly w.r.\ to the dimension $d$ .

The derivative bound \eqref{eq:utGevrxD} implies that
$v_d(\tau,\cdot)|_{I^d}$ is Gevrey-$\delta$-regular
with $\delta = 1/\min\{1, 2\rho\}$.
In particular, for $\delta = 1$, i.e.\ when $\rho\geq 1/2$, 
for every fixed $\tau>0$, $x\mapsto v_d(\tau,x)$ is real analytic in $I^d$,
which is the case we consider first.

In this case, 
we perform an affine change of coordinates to transform $v_d(\tau,\cdot)$ to the real analytic function $[-1,1]^d \ni \hat{x} \mapsto v_c(\tau,\hat{x})$. This function admits a holomorphic extension to some open set $O \subset \IC^d$ containing $[-1,1]^d$. By choosing $\bar{\varrho} > 1$ (the ``semiaxis sums'') sufficiently close to $1$, we obtain that $\cE_{\bar{\varrho}} \subset O$, i.e., $v_c(\tau,\cdot)$ admits a holomorphic extension to $\cE_{\bar{\varrho}}$, where the Bernstein polyellipse $\cE_{\bar{\varrho}} \subset \IC^d$ is defined as $d$-fold Cartesian product of the Bernstein ellipse $\{ (z+z^{-1})/2 \colon z \in \IC , 1 \leq |z| < \bar{\varrho}\}$. More precisely,  
$x\mapsto v_d(\tau,x)$ admits, 
with respect to each co-ordinate $x_i \in [x_-,x_+]$ of $x$,
a holomorphic extension to an open neighborhood of $[x_-,x_+]$ in $\IC$ 
(see, e.g., Krantz and Parks \cite[Section~1.2]{krantz1992primer}).
By Hartogs' theorem (see, e.g., H\"ormander \cite[Theorem~2.2.8]{HormIntroSevCplx}),
for every fixed $\tau>0$, 
$x\mapsto v_d(\tau,x)$ admits a holomorphic extension to a 
polyellipse in $\IC^d$ with foci at $x_-,x_+$ or,
in normalized coordinates 
\begin{equation}\label{eq:NormCoord}
\hat{x}_i = \big(T^{-1}(x)\big)_i = 2[x_i - (x_- + x_+)/2] / (x_+-x_-),\quad i=1,...,d,
\end{equation}
the map $[-1,1]^d \ni \hat{x} \mapsto v_d(\tau,T(\hat{x})) = v_c(\tau,\hat{x})$
admits a holomorphic extension to a Bernstein polyellipse $\cE_{\bar{\varrho}}\subset \IC^d$
with foci at $\hat{x}_i = \pm 1$, 
and semiaxis sums $1<\bar{\varrho} = \cO(A(\tau,\rho)^{-1})$. 
As $\tau\mapsto A(\tau,\rho)^{-1}$ is increasing for every fixed value of $\rho$,
for $\rho \geq 1/2$ parabolic smoothing increases the domain of holomorphy with $\tau$.

In the general case $\delta = 1/\min\{1,2\rho\}$ 
with $\rho>0$ as in \eqref{eq:StrEllPsi},
ReLU DNN expression rates of multivariate holomorphic (if $\rho \geq 1/2$) 
and Gevrey regular (if $0<\rho<1/2$)
functions such as $ \hat{x} \mapsto v_c(\tau,\hat{x})$ 
have been studied in Opschoor et al.\ \cite{OSZ19_839}. 

The holomorphy or Gevrey-$\delta$ regularity %[depending on $\delta$] 
of the map $ \hat{x} \mapsto {v_c}(\tau,\hat{x})$ implies, 
with Opschoor et al.\ \cite[Theorem 3.6, Proposition 4.1]{OSZ19_839}
that 
there exist constants $\beta'=\beta'(\bar{\varrho},d)>0$ and $C = C(u_d,\bar{\varrho},d) > 0$,
and for every $\cN\in\N$
there exists a ReLU DNN $\tilde u_\cN:[-1,1]^d\to\R$ such that
\begin{equation}\label{eq:findimsize}
M(\tilde{u}_\cN) \leq \cN , 
\qquad
L(\tilde{u}_\cN) \le C \cN^{\min\{\frac{1}{2}, \frac{1}{d+1/\delta}\}} \log(\cN) 
\end{equation}
and such that the error bound 
\begin{align}
\label{eq:findimrate}
\|{v_c}(\tau,\cdot) - \tilde{u}_{\cN}(\cdot)\|_{W^{1,\infty}([-1,1]^d)}
\leq C\exp\left( -\beta' \cN^{\min \{ \frac{1}{2\delta}, \frac{1}{\delta d+1}\}} \right) 
\end{align}
holds.
Reverting the affine change of variables \eqref{eq:NormCoord} in the input 
layer, we obtain the following result on the $\varepsilon$-complexity 
of the ReLU DNN expression error for $x\mapsto v_d(\tau,x)$ at fixed $0<\tau\leq T$.
%%%%%%%%%%%%%%%%%%%%%%%%%%%%%%%%%%%%%%%%%%%%%%%%%%%%%%%%%%%%%%%%%%%%%%%%%%%%%%%%%%%%%%%%%%%5
\begin{theorem}\label{thm:DNNExprd}
	Assume that the symbol
	$\psi$ of the $\R^d$-valued LP $X^d$ 
	satisfies \eqref{eq:StrEllPsi} 
	with $|\xi|^2 = \xi^\top \xi$ and
	with some $\rho \in (0,1]$.
	
	Then,
	for every $\varphi_d$ with $v_d(0,\cdot) = \varphi_d \circ \exp  \in L^2(\R^d)$, 
	for every $\tau>0$, 
	on every closed, bounded hypercube 
	$I^d = [x_-,x_+]^d\subset \IR^d$ 
	and, respectively, 
	$J^d = [s_-,s_+]^d \subset (0,\infty)^d$ with $s_\pm = \exp(x_\pm)$ 
	the 
	variational solutions $v_d$ of the Kolmogorov PIDE \eqref{eq:MultdPDEx} at $\tau$
	and $u_d(\tau,s) = v_d(\tau, \log(s))$
	can be expressed on $I^d$, $J^d$ 
	by ReLU DNNs $\tilde{v}_{d,\varepsilon}$, $\tilde{u}_{d,\varepsilon}$
	at exponential rate.
	
	Specifically, 
	there exists a constant 
	$C=C(x_-,x_+,\delta,d,\tau)> 0$ 
	such that, with $\delta = 1/\min\{ 1, 2\rho \} \geq 1$,  
	for every $0<\varepsilon \leq 1/2$ 
	exist ReLU DNNs $\tilde{v}_{d,\varepsilon}$, $\tilde{u}_{d,\varepsilon}$
	for which there holds
	$$
	\sup_{x \in I^d}| v_d(\tau,x) - \Rl(\tilde{v}_{d,\varepsilon})(x) |
	,\;\;
	\sup_{s \in J^d}| u_d(\tau,s) - \Rl(\tilde{u}_{d,\varepsilon})(s) |
	\leq \varepsilon\;,
	$$
	and, 
	\[\begin{aligned}
	M(\tilde{v}_{d,\varepsilon}) + M(\tilde{u}_{d,\varepsilon})
	& \leq 
	C |\log(\varepsilon)|^{\max\{ 2\delta, \delta d+1\} }\;,
	\\  
	L(\tilde{v}_{d,\varepsilon}) + L(\tilde{u}_{d,\varepsilon})
	& \leq 
	C |\log(\varepsilon)|^\delta |\log(|\log(\varepsilon)|)|\;.
	\end{aligned} \]
	Here, the constants $C = C(\delta,d,\tau) > 0$ depend on $I$ and $J$ and,
	generally, exponentially on the basket size $d$.
\end{theorem}
\begin{proof} 
	The asserted bounds for $\Rl(\tilde{v}_{d,\varepsilon})$ 
	follow by elementary manipulations 
	from insisting that the expression error bound
	\eqref{eq:findimrate} equal $\varepsilon \in (0,1/2]$ 
	and subsequently inserting the resulting expression 
	$
	\cN \simeq | \log(\varepsilon) |^{\max\{ 2\delta, \delta d+1 \}}
	$
	into the bounds \eqref{eq:findimsize} for the DNN size and depth.
	
	The bounds for $\Rl(\tilde{u}_{d,\varepsilon})$ are then deduced
	from those for $\Rl(\tilde{v}_{d,\varepsilon})$ and the fact that
	the transformation $\log(\cdot): J^d \to I^d$ 
	(understood component-wise) is real analytic. Hence, it admits 
	a holomorphic extension to an open neighbourhood of $J^d$ in $\C^d$.
	Then Opschoor et al.\ \cite[Theorem~3.6]{OSZ19_839}, 
	combined with the affine transformation $T: [-1,1]^d \to J^d$, 
	implies that there are constants $C,\beta'>0$ 
	such that for every $\cN \in \IN$
	exists a ReLU DNN $\widetilde{\log}_\cN$ 
	such that
	\begin{equation*}
	M\big(\widetilde{\log}_\cN(\cdot)\big) \leq \cN
	, \qquad
	L\big(\widetilde{\log}_\cN(\cdot)\big) \le C \cN^{\frac{1}{d+1}} \log_2(\cN)
	\end{equation*}
	and the error bound
	\begin{equation} \label{eq:findimratelog}
	\normc[{W^{1,\infty}([-1,1]^d)}]{(\log\circ T)(\cdot) - \Rl(\widetilde{\log}_{\cN})\circ T(\cdot)}
	\leq 
	C\exp\left( -\beta' \cN^{\frac{1}{d+1}} \right)\;.
	\end{equation}
	For every $\cN\in\N$ 
	the set $\widetilde{I^d} = \Rl(\widetilde{\log_\cN})(J^d)\cup \log(J^d) \subset (-\infty,\infty)^d$
	is compact due to \eqref{eq:findimratelog}.
	For given $\varepsilon \in (0,1/2]$, we choose $\cN\in\N$ as before. 
	Using that $\cN^{\min\{\frac{1}{2\delta},\frac{1}{\delta d+1}\}} \leq \cN^{\frac{1}{d+1}}$, 
	this choice guarantees that in \eqref{eq:findimratelog} it holds that 
	$C\exp( -\beta' \cN^{\frac{1}{d+1}}) \leq \varepsilon$.
	Then we define 
	$\widetilde{u}_d(\tau,\cdot) = \Rl(\tilde{v}_{d,\varepsilon})(\cdot) \circ \Rl(\widetilde{\log_\cN})(\cdot)$
	and estimate
	$$
	\begin{array}{rcl}
	\sup_{s \in J^d}| u_d(\tau,s) & - & \widetilde{u}_d(\tau,s) |
	\\ & = & \displaystyle
	\sup_{s \in J^d}| v_d(\tau,\cdot) \circ \log(s) - \widetilde{u}_d(\tau,s)  |
	\\
	& \leq & \displaystyle 
	\sup_{s \in J^d}|  v_d(\tau,\cdot) \circ \log(s) -  v_d(\tau,\cdot) \circ \Rl(\widetilde{\log_\cN})(s) |
	\\
	& & \displaystyle 
	+ \sup_{s \in J^d}|  v_d(\tau,\cdot) \circ \Rl(\widetilde{\log_\cN})(s) 
	- \Rl(\tilde{v}_{d,\varepsilon})(\cdot) \circ \Rl(\widetilde{\log_\cN})(s) | 
	\\
	& \leq & \displaystyle
	\| v_d(\tau,\cdot) \|_{W^{1,\infty}(\widetilde{I^d})} 
	\sup_{s \in J^d}| \log(s) - \Rl(\widetilde{\log}_{\cN})(s) |
	\\
	& & \displaystyle
	+ \sup_{x \in \widetilde{I^d}}|  v_d(\tau,x) 
	- \Rl(\tilde{v}_{d,\varepsilon})(x) |
	\\
	& \leq & \displaystyle
	C \varepsilon \;.
	\end{array}
	$$
	Since the DNN size and DNN depth are additive under composition of ReLU DNNs,
	the assertion for $\tilde{u}_{d,\varepsilon}$ follows (possibly adjusting the value
	of the constant $C$).
\end{proof}
\begin{remark}
	Some sufficient conditions on the characteristic triplet 
	$(A^d,\gamma^d,\nu^d)$ that ensure \eqref{eq:StrEllPsi} 
	in the multivariate setting are as follows.
	Consider first the case 
	when the diffusion component is non-degenerate, i.e. 
	$A^d$ is positive definite. 
	Then
	\[ 
	\begin{aligned}
	\Re \psi(\xi) & = \frac{1}{2} \xi^\top A^d \xi  - \int_{\R^d } 
	\left(\cos(\xi^\top y)-1\right) \nu^d(d y) 
	\geq C_1 | \xi |^{2}
	\\ | \psi(\xi) | & \leq \frac{1}{2} |\xi^\top A^d \xi| + |\xi^\top \gamma^d|  + \int_{\{\|y\|\leq 1\}} 
	\left|e^{i \xi^\top y}-1- i \xi^\top y \right| \nu^d(d y)
	\\ & \quad \quad + 2 \int_{\{\|y\|> 1\}} \nu^d(d y) \leq C_2 | \xi |^{2} + C_3
	\end{aligned}
	\] 
	for suitable choices of $C_1,C_2,C_3 > 0$. 
	\newline
	In the case when $A^d$ is not positive definite, 
	we refer for instance to Eberlein and Glau \cite[Section~7]{EberleinGlauVarSlnLevyPIDEs} 
	and Hilber et al.\ \cite[Lemma~14.5.1]{CMQFBook} for sufficient conditions. 
\end{remark}
%%%%%%%%%%%%%%%%%%%%%%%%%%%%%%%%%%%%%%%
\subsection{Breaking the Curse of Dimensionality}
\label{sec:LevyCrsD}
%%%%%%%%%%%%%%%%%%%%%%%%%%%%%%%%%%%%%%%
The result Theorem \ref{prop:Ddresult} demonstrated 
$\eps$ expression error for DNNs whose depth and size 
are bounded polynomially in terms of $\eps^{-1}d$,
for European style options in multivariate, exponential L\'evy models. 
%with constants that depend polynomially on the number $d$ of assets.
In particular, in Theorem \ref{prop:Ddresult} 
the curse of dimensionality was proved to be overcome for a market model with jumps:
a DNN expression rate was shown that is algebraic 
in terms of the target accuracy $\eps>0$ 
with constants that depend polynomially on the dimension $d$.
The rates $\mathfrak{p},\mathfrak{q} \in [0,\infty)$ can be read
off the proof of Theorem \ref{prop:Ddresult}; however, these constants could
be large, thereby affording only low DNN expression rates.

Theorem \ref{thm:DNNExprd}, 
on the other hand, stated \emph{exponential expressivity} 
of deep ReLU NNs, i.e.\  maximum expression error at time $\tau>0$ 
with accuracy $\varepsilon > 0$ can be attained 
by a deep ReLU NN of size and depth which grow polylogarithmically
with respect to $|\log(\varepsilon)|$. 
This exponential expression rate bound was, however, 
still prone to the curse of dimensionality.

In the present section we further address alternative mathematical arguments on
how DNNs can overcome the CoD in the presently considered
jump-diffusion models. Specifically,
two mathematical arguments in addition to the probabilistic
arguments in Section \ref{sec:DNNProbab} are presented.
Both exploit stationarity of the LP $X^d$
which implies \eqref{eq:hatvdxi}, \eqref{LKd}, to 
obtain DNN expression rates free from the curse of dimensionality.
\subsubsection{Barron Space Analysis} 
\label{sec:Barron}
The first alternative approach to Theorem \ref{prop:Ddresult} is based on verifying, 
using \eqref{eq:hatvdxi}, \eqref{LKd}, 
regularity of option prices in the so-called \emph{Barron space} 
introduced in the fundamental work Barron \cite{Barron1993}. 
It will provide DNN expression error
bounds with explicit values for $\mathfrak{p}$ and $\mathfrak{q}$, however, 
\emph{in \cite{Barron1993} only for DNNs with sigmoidal activation functions $\varrho$}; 
similar
results for ReLU activations are asserted in E and Wojtowytsch \cite{e2020observations}.
For simplicity, we consider here a subset $\cB$ of Barron space. 
An integrable function $f:\IR^d\to \IR$ belongs to $\cB$ if 
\begin{equation}\label{eq:DefB1}
\|f\|_{\cB} = \int_{\IR^d}  |\xi| |\hat{f}(\xi)| d\xi < \infty\;.
\end{equation}
The explicit appearance of the Fourier transform $\hat{f}$ renders the norm $\| \circ \|_{\cB}$
in \eqref{eq:DefB1} particularly suitable for our purposes due to \eqref{eq:hatvdxi}-\eqref{eq:symbolD}.
As was pointed out in \cite{Barron1993,e2020observations},
the relevance of the Barron norm 
$\| \circ \|_{\cB}$ stems from it being sufficient
for dimension-robust DNN approximation rates.
For $m\in \IN$, consider the two-layer neural networks 
$f_m$ which are given by
\begin{equation}\label{eq:2LNN}
f_m:\IR^d\to \IR, \, x\mapsto \frac{1}{m} \sum_{i=1}^m a_i \varrho(w_i^\top x + b_i)
\end{equation}
with parameters $(a_i,w_i,b_i)\in \IR \times \IR^d \times \IR$.
Their relevance stems from the following result:  
assume that $\varrho$ is sigmoidal, i.e., 
bounded, measurable 
and $\varrho(z) \to 1$ as $z \to \infty$, $\varrho(z) \to 0$ as $z \to -\infty$. 

Then,
for $f\in \cB$ and for every $R>0$, $d\in \IN$, 
and for every $m\in \IN$ 
exist parameters $\{ (a_i,w_i,b_i) \}_{i=1}^m$ such that 
for the corresponding DNN $f_m$ as in \eqref{eq:2LNN} holds
\begin{equation}\label{eq:BDNNErr}
\begin{array}{c}
\| f - f_m \|_{L^2([-R,R]^d;\pi)} \leq \max\{1,R\} m^{-1/2} \| f \|_{\cB}\;.
\end{array}
\end{equation}
Here, $\pi$ denotes a probability measure charging $[-R,R]^d$.
The bound \eqref{eq:BDNNErr} follows from \cite[Theorem~1]{Barron1993}, 
and was generalized in \cite[Eqn. (1.3)]{e2020observations} 
to ReLU activation.
%We mention that the norm $\| \circ \|_{\cB}$ in \eqref{eq:DefB1} 
%could be weakened \cite{e2020observations,e2020mathematical}.
%In view of the Fourier representation \eqref{eq:hatvdxi} 
%of $v_d(\cdot,t)$ in terms of the characteristic function $\psi$
%of the LP $X^d$, \eqref{eq:DefB1} is the most suitable norm 
%in the present context. 

The bound in \eqref{eq:BDNNErr} is free from the CoD: 
the number $N$ of parameters in the DNN grows as $\cO(md)$ so that 
$m^{-1/2} \leq CN^{-1/2} d^{1/2}$ with absolute constant $C>0$.

With \eqref{eq:hatvdxi}, \eqref{LKd}, for every $\tau\geq 0$,
sufficient conditions for $x\mapsto v_d(\tau,x)$ to belong to $\cB$
can be verified. 
With \eqref{eq:BDNNErr}, DNN mean square expression rate bounds 
of option prices that are free from the CoD follow.
\begin{proposition}\label{prop:BarronLevy}
	Assume that $\varrho = \ReLU$.
	Assume furthermore that
	the payoff in log-variables, $v_d(0,\cdot)$, belongs to $\cB$.
	
	Then, for every $\tau\geq 0$, the price $x\mapsto v_d(\tau,x)$ 
	can be expressed by a NN $x\mapsto \tilde{v}_d(\tau,x)$ 
	of depth $2$ and size $m(d+2)$ with 
	$m\in \IN$ and error bound
	\begin{equation*}
	\| v_d(\tau,\cdot) - \tilde{v}_d(\tau,\cdot) \|_{L^2([-R,R]^d{;\pi})} 
	\leq 
	\max\{1,R\} m^{-1/2} \| v_d(0 , \cdot) \|_{\cB}
	\;. 
	\end{equation*}
	Here, $\pi$ denotes a probability measure on $[-R,R]^d$.
\end{proposition}
\begin{proof} 
	We observe that for every $\xi\in \IR^d$ the identity \eqref{LKd} with $\tau=1$ shows that 
	\[ 
	\exp\big(- \Re \psi_{X^d}(\xi)\big)= |\exp\big(-\psi_{X^d}(\xi)\big)| = |\IE[\exp(i\xi^\top X_1^d)] | \leq 1
	\]
	and therefore $\Re \psi_{X^d}(\xi) \geq 0$. 
	From \eqref{eq:hatvdxi}, \eqref{LKd} we obtain for $\tau\geq 0$
	$$
	\forall \xi\in \IR^d: \; 
	|\hat{v}_d(\tau,\xi)| 
	= 
	|\exp\big(-\tau\psi_{X^d}(\xi)\big)\hat{v}_d(0,\xi)|
	=  
	\exp\big(-\tau\Re \psi_{X^d}(\xi)\big)|\hat{v}_d(0,\xi)|.
	$$
	The payoff in log-price, $v_d(0,\xi)$, belonging to $\cB$ 
	implies $\| v_d(0,\cdot) \|_{\cB} <\infty$. 
	Using $\Re \psi_{X^d}(\xi) \geq 0$, we find
	for every $\tau \geq 0$, $\xi\in \IR^d$ that
	$|\hat{v}_d(\tau,\xi)| \leq |\hat{v}_d(0,\xi)|$.
	This implies that for every $\tau \geq 0$, 
	$\| v_d(\tau,\cdot) \|_{\cB} \leq \| v_d(0,\cdot) \|_{\cB}$.
	The approximation bound \eqref{eq:BDNNErr} implies the assertion.
\end{proof}
Pointwise, $L^\infty$-norm error bounds can be obtained using \cite[Eqn. (1.4)]{e2020observations}.
\subsubsection{Parabolic Smoothing and Sparsity of Chaos Expansions}
\label{sec:ParSmoot}
The second non-probabilistic approach to Theorem \ref{prop:Ddresult} 
towards DNN expression error rates not subject to the CoD
is based on dimension-explicit derivative
bounds of option prices, which allow in turn to establish 
summability bounds for generalized polynomial chaos (gpc for short)
expansions of these prices. 
Good summability of gpc coefficient sequences is well known
to imply high, dimension-independent rates of approximation
by sparse, multivariate polynomials.
This, in turn, implies corresponding expression rates 
by suitable DNNs Schwab and Zech \cite[Theorem~3.9]{SZ19_2592}. 
Key in this approach is to exploit \emph{parabolic smoothing}
of the Kolmogorov PDE. The corresponding dimension-independent 
expression rate results will generally be higher 
than those based on probabilistic or 
Barron space analysis, but will hold only 
\emph{for sufficiently large $\tau>0$}.

We start by discussing
more precisely the dependence of the constants
in the proof of Theorem \ref{thm:DNNExprd} on the dimension $d$.
\begin{remark}  \label{rmk:Crseofd-2}
	The constant $C(d)=C(\tau,d)$ in the derivative bound \eqref{eq:utGevrxD}
	need not be exponential in $d$. 
	To see it, we bound \eqref{eq:utGevrxD} by the 
	inverse Fourier transform and the Cauchy-Schwarz inequality. 
	For 
	$\alpha \in \IN_0^d$ with $|\alpha| = \sum_{i=1}^d \alpha_i = k$,
	we find with the Cauchy-Schwarz inequality
	and with the lower bound \eqref{eq:StrEllPsi}
	\begin{equation*}
	\begin{aligned}
	\sup_{x \in I^d}|  & (D^\alpha_x v_d)  (\tau,x) |
	\\ & = \sup_{x \in I^d}
	\left|
	\frac{1}{(2\pi)^{d/2}} \int_{\R^d} 
	(i\xi)^\alpha \exp(i x^\top \xi) \exp\big(-\tau\psi(\xi)\big) \hat{v}_d(0,\xi) d \xi 
	\right|
	\\
	& \leq \frac{1}{(2\pi)^{d/2}} \int_{\R^d} |\xi|^k |\exp\big(-\tau\psi(\xi)\big)| |\hat{v}_d(0,\xi)| d \xi 
	\\ 
	& \leq \frac{1}{(2\pi)^{d/2}} \left( \int_{\R^d} \exp(-2 \tau C_1 |\xi|^{2\rho}) d \xi \right)^{1/2}  
	\\ & \quad \quad \cdot \left( \int_{\R^d} |\xi|^{2k} \exp(-2 \tau C_1 |\xi|^{2\rho}) |\hat{v}_d(0,\xi)|^2 d \xi \right)^{1/2}
	.
	\end{aligned}
	\end{equation*}
	The last factor can be bounded precisely by the square-root of the right hand side of 
	\eqref{eq:L2Estd} (by using  \eqref{eq:maxexp}).
	Using $k^k \leq k!e^k$ we obtain the bound \eqref{eq:utGevrxD} 
	as
	\begin{equation}\label{eq:auxEq30}
	\sup_{x \in I^d}|  (D^\alpha_x v_d)(\tau,x) |
	\leq
	C(d,\tau) \big(A(\tau,\rho)\big)^k (k!)^{1/\min\{1, 2\rho\}} \| v_d(0,\cdot) \|_{L^2(\IR^d)} \;\end{equation}
	with constant
	$A(\tau,\rho) =  (2\tau C_1\rho)^{-1/(2\rho)}$ 
	and the explicit constant
	\begin{align}
	C(d,\tau)
	&  \nonumber
	= \frac{1}{(2\pi)^{d/2}} \left( \int_{\R^d} \exp(-2 \tau C_1 |\xi|^{2\rho}) d \xi \right)^{1/2} 
	\\ \nonumber& = 
	\frac{1}{(2\pi)^{d/2}} \left(2 \frac{\pi}{d} \omega_d \int_0^\infty r^{d-1} \exp(-2 \tau C_1 r^{2\rho}) d r \right)^{1/2}
	\\  \nonumber
	& 
	= \frac{1}{(2\pi)^{d/2}} 
	\left( \frac{ \pi}{ \rho d} \frac{1}{(2\tau C_1)^{d/(2\rho)}} \omega_d \Gamma(\frac{d}{2\rho})
	\right)^{1/2} 
	\\ & =
	\frac{1}{(2\pi)^{d/2}} 
	\left( \frac{ \pi}{ \rho d} \frac{1}{(2\tau C_1)^{d/(2\rho)}} \frac{\pi^{d/2} 
		\Gamma(\frac{d}{2\rho})}{\Gamma(\frac{d}{2}+1)}\right)^{1/2},
	\label{eq:Cexplicit}
	\end{align} 
\end{remark}
where $\omega_d$ denotes the volume of the unit ball in $\IR^d$.
Inspecting the constant $C(d,\tau)$ in \eqref{eq:Cexplicit}, 
we observe that e.g.\ for 
$\rho = 1$ and $\tau_0 = \tau_0(C_1) = 1/(8\pi C_1)$, 
$\tau \geq \tau_0 > 0$ sufficiently large implies that
the constant $C(d,\tau)$ is bounded independent of $\tau$ and $d$.
\begin{remark}\label{rmk:Crseofd-1}
	In certain cases, the parabolic smoothing 
	implied by the ellipticity assumption \eqref{eq:StrEllPsi} on the generator $\cA$
	entails that the constant $C$ in the regularity estimates 
	\eqref{eq:utGevrxD}
	grows only polynomially with respect to $d$. 
	For instance, in  Remark~\ref{rmk:Crseofd-2} we provided 
	sufficient conditions which ensure that the constant  $C$ in the regularity estimates 
	\eqref{eq:utGevrxD} is even bounded with respect to $d$.
	This allows to derive an explicit and dimension-independent
	bound on the series of Taylor coefficients. 
	This, in turn, allows to obtain bounds on the constant in \eqref{eq:findimrate}
	which scale polynomially with respect to $d$.
	Consider, for example, $\rho=1$ (i.e.\ non-degenerate diffusion) 
	and \emph{assume that $\tau>0$ is sufficiently large}:
	specifically, $(2 \tau C_1)^{1/(2\rho)} \geq 1$ and $ d A(\tau,\rho) < 1$, 
	where $A(\tau,\rho) = (2\rho \tau C_1)^{-1/(2\rho)}$ 
	denotes the constant in large parentheses of \eqref{eq:GevEst}.
	This holds if
	\begin{equation}\label{eq:t>}
	\tau > \frac{d^{2\rho}}{2\rho C_1} \;.
	\end{equation}
	With \eqref{eq:t>} and 
	using $ \sum_{\alpha \in \N_0^d, |\alpha|=k} {k \choose \alpha} = d^k$,
	we may estimate with the multinomial theorem
	\begin{align*}
	\sum_{\alpha \in \N_0^d}
	\frac{ \sup_{x \in I^d}|  (D^\alpha_xv_d)(\tau,x) |}{\alpha!} 
	&\leq 
	C(d,\tau) \sum_{\alpha \in \N_0^d} \frac{A(\tau,\rho)^{|\alpha|} (|\alpha|)!}{\alpha!} 
	\\
	&= 
	C(d,\tau) \sum_{k=0}^\infty \big(d A(\tau,\rho)\big)^{k} = C(d,\tau) \frac{1}{1-d A(\tau,\rho)}. 	
	\end{align*}
	By Remark \ref{rmk:Crseofd-2},
	\eqref{eq:t>} implies that $C(d,\tau)$ in \eqref{eq:Cexplicit} 
	is bounded uniformly with respect to $d$. 
	Thus, in this case one may obtain bounds on the constant in \eqref{eq:findimrate} 
	which scale polynomially with respect to $d$. 
	However, the DNN size still grows polylogarithmically with respect to the dimension $d$, 
	in terms of $|\log(\varepsilon)|$ (i.e., at least as $\cO(|\log(\varepsilon)|^d)$),
	so that the curse of dimensionality is not overcome.
\end{remark}
The constant $C>0$ in the exponential expression rate bounds established
in Theorem \ref{thm:DNNExprd}
depends in general exponentially on the basket size $d$, 
resp.\ on the dimension of the 
solution space of the PIDE \eqref{eq:MultdPDEx}, due to the reliance on the
ReLU DNN expression rate analysis in Opschoor et al.\ \cite{OSZ19_839}.
Furthermore, the DNN size grows polylogarithmically with respect to 
the dimension $d$, in terms of $|\log(\varepsilon)|$.
Considering exponential expression rate bounds, 
this exponential dependence on $d$ in terms of $|\log(\varepsilon)|$
seems, in general, not avoidable,
as can be seen from \cite[Theorem~3.5]{OSZ19_839}.
Nevertheless, in Remark \ref{rmk:Crseofd-1} we already 
hinted at parabolic smoothing implying sufficient regularity 
(under the $d$-dependent provision \eqref{eq:t>} on $\tau$)
for polynomial w.r.\ to $d$ constants in DNN expression rate bounds.

In the following paragraphs, 
we settle for \emph{algebraic DNN expression rates}
and overcome exponential dependence on $d$ in ReLU DNN expression
error bounds under certain \emph{sparsity assumptions on polynomial chaos expansions}, 
as shown in Schwab and Zech \cite{SZ19_2592}, Cohen et al.\ \cite{CDS1} and the references there.
We develop a variation of the results in \cite{SZ19_2592} 
in the present context. 

We impose the following hypothesis, which takes the place of 
the lower bound in \eqref{eq:StrEllPsi}. 
We still impose $| \psi(\xi) | \leq C_2 | \xi |^{2\rho} + C_3$, i.e., 
the second condition in \eqref{eq:StrEllPsi} holds for each $d \in \N$ 
(but $C_2$, $C_3$ and $\rho$ in that condition are allowed to depend on $d$).
%%%%%%%%%%%%%%%%%%%%%%%%%%%%%%%%%%%%%%%%%%%%%%%%%%%%%%%%%%%%%%%%%%%%%%%%%%%%%%%%%5
\begin{assumption}\label{ass:Anisotr}
	There exists a constant $C_1 >0$ and $(\rho_j)_{j \in \IN}$ with $\frac{1}{2}<\rho_j\leq 1$, 
	such that for each $d \in \IN$, 
	the symbol $\psi_{X^d}$  of the LP $X^d$ satisfies that
	\begin{equation}\label{eq:AdAnis}
	\forall \xi \in \IR^d: 
	\quad 
	\Re\psi_{X^d}(\xi) \geq C_1 \sum_{j=1}^d |\xi_j|^{2\rho_j}
	\;.
	\end{equation}
	Furthermore, 
	\begin{equation*} \rho= \inf_{j \in \IN} \rho_j > \frac{1}{2}.
	\end{equation*}
	The payoff function $\varphi_d$ in \eqref{eq:MultdPDEx} is such that
	$v_d(0,\cdot)= \varphi_d \circ \exp \in L^2(\IR^d)$.
\end{assumption}
In comparison to the lower bound in \eqref{eq:StrEllPsi} the condition \eqref{eq:AdAnis} is restricted to the case $\rho>\frac{1}{2}$. On the other hand, different exponents $\rho_j$ are allowed along each component. Furthermore, note that Assumption~\ref{ass:Anisotr} imposes that $C_1$ does not depend on the dimension $d$.
%%%%%%%%%%%%%%%%%%%%%%%%%%%%%%%%%%%%%%%%%%%%%%%%%%%%%%%%%%%%%%%%%%%%%%%%%%%%%%%%5
\begin{remark}\label{rmk:SuffCond}
	Consider the pure diffusion case, i.e., 
	when the characteristic triplet is $(A^d,0,0)$ with a symmetric, positive definite diffusion matrix $A^d$ and 
	L\'evy-symbol $\psi_{X^d}: \IR^d\to \IR: \xi \mapsto \xi^\top A^d \xi$.
	
	A sufficient condition for assumption \eqref{eq:AdAnis} to hold is that 
	the eigenvalues $(\lambda_i^d)_{i=1,\ldots,d}$ of $A^d$ be lower bounded away from zero, 
	\begin{equation}\label{eq:auxEq31} 
	C_1 = \inf_{i,d} \lambda_i^d > 0. 
	\end{equation}
	To see this, 
	write $Q^\top A^d Q = D$ for a diagonal matrix $D$ 
	containing the eigenvalues of $A^d$ and an orthogonal matrix $Q$. 
	Then we obtain for arbitrary $\xi\in \IR^d$
	\[\begin{aligned}
	\psi_{X^d}(\xi) = \xi^\top A^d \xi 
	= 
	(\xi^\top) Q D Q^\top \xi 
	= 
	\sum_{i=1}^d \lambda_i (Q^\top \xi)_i^2 
	& \geq 
	\left(\min_i{\lambda_i} \right) |Q^\top \xi|^2 
	\\ & = 
	\left(\min_i{\lambda_i} \right) |\xi|^2 .
	\end{aligned}\]  
	Therefore condition \eqref{eq:AdAnis} is satisfied 
	with $C_1$ as in \eqref{eq:auxEq31} and $\rho_j=1$ for all $j \in \IN$.

	This condition imposes, in applications, that different assets 
	(modelled by different components of the LP $X^d$) 
	should not become asymptotically (perfectly) dependent as the dimension grows.
\end{remark}

\begin{remark}
	Consider characteristic triplets $(A^d,\gamma^d,\nu^d)$ and the more general case of non-degenerate diffusion, 
	i.e.\ with $A^d$ satisfying the condition \eqref{eq:auxEq31} formulated in Remark~\ref{rmk:SuffCond}. 
	Then the real part of the L\'evy symbol $\psi_{X^d}$ of $X^d$ satisfies for all $\xi \in \R^d$
	\[
	\Re\psi_{X^d}(\xi) = 
	\frac{1}{2} \xi^\top A^d \xi - \int_{\R^d } 
	\left(\cos(\xi^\top y)-1 \right) \nu^d(d y) \geq \frac{1}{2} \xi^\top A^d \xi \geq C_1 |\xi|^2  
	\]
	with $C_1$ as in \eqref{eq:auxEq31}. 
	Hence, Assumption~\ref{ass:Anisotr} is satisfied also in this more general situation. 
	Further examples of LP satisfying Assumption~\ref{ass:Anisotr} 
	are based on stable-like processes and copula-based constructions as e.g.\ in Farkas et al.\ \cite{FRS07}.
\end{remark}
%%%%%%%%%%%%%%%%%%%%%%%%%%%%%%%%%%%%%%%%%%%%%%%%%%%%%%%%%%%%%%%%%%%%%%%%%%%%%%%%%%%%%%%%%
%%%%%%%%%%%%%%%%%%%%%%%%%%%%%%%%%%%%%%%%%%%%%%%%%%%%%%%%%%%%%%%%%%%%%%%%%%%%%%%%%%%%%%%%%

As we shall see below, Assumption \ref{ass:Anisotr} 
ensures good ``separation'' and ``anisotropy'' properties
of the symbol \eqref{eq:symbolD}
of the corresponding L\'evy process $X^d$.

For $\tau>0$ satisfying \eqref{eq:t>}, 
we analyze the regularity of $x\mapsto v_d(\tau,x)$.
From Assumption \ref{ass:Anisotr} we find that for every $\tau>0$,
$x\mapsto v_d(\tau,x)\in L^2(\IR^d)$
and that 
its Fourier transform has the explicit form
\begin{equation}\label{eq:FTvdt}
\hat{v}_d(\tau,\xi) 
= F_{x\to\xi}v_d(\tau,\cdot) 
= \exp\big(-\tau \psi_{X^d}(\xi)\big) \hat{v}_d(0,\xi)\;.
\end{equation}
For a multi-index $\bsnu=(\nu_1,...,\nu_d)\in \IN_0^d$, 
denote by $\partial_x^\bsnu$ the mixed partial derivative 
of total order $|\bsnu|=\nu_1+...+\nu_d$ with respect to $x\in \IR^d$.
Formula \eqref{eq:FTvdt} and Assumption \ref{ass:Anisotr} can be used to show
that for every $\tau>0$, $x\mapsto v_d(\tau,x)$ is analytic at any $x\in \IR^d$.
This is of course the well-known smoothing property of the 
generator of certain non-degenerate L\'evy processes.
To address the curse of dimensionality, we quantify the smoothing
effect in a $d$-explicit fashion.

To this end,
with Assumption \ref{ass:Anisotr}
we calculate for any $\bsnu\in \IN^d_0$ at $x=0$ (by stationarity,
the same bounds hold for the Taylor coefficients at any $x\in \IR^d$)
\[\begin{aligned}
(2\pi)^{d/2}
|\partial^\bsnu_xv_d(\tau,0)|
& = 
\left| \int_{\xi\in \IR^d} (i\xi)^\bsnu \hat{v}_d(\tau,\xi) d\xi  \right|
\\ & \leq 
\int_{\xi\in \IR^d} 
|\hat{v}_d(0,\xi)| \prod_{j=1}^d |\xi_j|^{\nu_j} 
\exp(-\tau C_1 |\xi_j|^{2 \rho_j}) d\xi\;.
\end{aligned} \]
We use \eqref{eq:maxexp} with 
$m\leftarrow \nu_j$, $\kappa \leftarrow C_1 \tau$, $\mu=2 \rho_j$
to bound the product as
$$
\prod_{j=1}^d |\xi_j|^{\nu_j} \exp(-\tau C_1 |\xi_j|^{2 \rho_j})
\leq 
\prod_{j=1}^d \left( \frac{\nu_j}{2 \rho_j \tau C_1  e} \right)^{\nu_j/(2\rho_j)} 
\;.
$$
We arrive at the following bound for the 
Taylor coefficient of order $\bsnu \in \IN_0^d$
of $v_d(t,\cdot)$ at $x=0$:
\begin{align}\nonumber 
|t_\bsnu| 
& = \left| \frac{1}{\bsnu!} \partial^\bsnu_x v_d(\tau,x)\mid_{x=0} \right|
\\ & \label{eq:TaylBd} \leq 
\frac{1}{(2\pi)^{d/2}} \| \hat{v}_d(0,\cdot) \|_{L^1(\IR^d)} 
\prod_{j= 1}^d \frac{1}{\nu_j!} \left( \frac{\nu_j}{2\rho_j \tau C_1 e} \right)^{\nu_j/(2\rho_j)} 
\;.
\end{align}
Stirling's inequality
\begin{equation*}
\forall n\in \IN: \quad n! \geq n^ne^{-n}\sqrt{2\pi n}  \geq n^ne^{-n}
\end{equation*}
implies in \eqref{eq:TaylBd} the bound 
\begin{equation}\label{eq:TaylBd2}
\forall \bsnu\in \IN_0^d: \quad 
|t_\bsnu|
\leq 
\frac{1}{(2\pi)^{d/2}} 
\| \hat{v}_d(0,\cdot) \|_{L^1(\IR^d)} \left({(\bsnu!)^{-1}}  \bsb^\bsnu \right)^{\rho'}\;.
\end{equation}
Here, $\rho' = 1-\frac{1}{2\rho} > 0$ and the positive weight sequence $\bsb = (b_j)_{j\geq 1}$ is given by 
$b_j= (2 \rho_j \tau C_1)^{-1/(2\rho_j \rho')}$, $j=1,2,...$ and multi-index notation is employed:
$\bsnu^{-\bsnu} = (\nu_1^{\nu_1}\nu_2^{\nu_2}...)^{-1}$, 
$\bsb^\bsnu=b_1^{\nu_1}b_2^{\nu_2}...$ and $\bsnu! = \nu_1!\nu_2!...$, 
with the convention $0!=1$ and $0^0 = 1$.

We raise \eqref{eq:TaylBd2}
to a power $q>0$, with $q < 1/\rho'$ and 
sum the resulting inequality over all $\bsnu\in \IN_0^d$
to estimate (generously)
\begin{align*}
\sum_{\bsnu\in \IN_0^d} |t_\bsnu|^q
& \leq 
\frac{\|\hat{v}_d(0,\cdot) \|_{L^1(\IR^d)}^q }{(2\pi)^{dq/2}} 
\sum_{\bsnu\in \IN_0^d}
\left( \frac{1}{\bsnu!}\bsb^\bsnu \right)^{q \rho'}
\\ & \leq \frac{\|\hat{v}_d(0,\cdot) \|_{L^1(\IR^d)}^q }{(2\pi)^{dq/2}} 
\sum_{\bsnu\in \IN_0^d}
\left( \frac{|\bsnu|!}{\bsnu!}\bsb^\bsnu \right)^{q \rho'}
\;.
\end{align*}
To obtain the estimate \eqref{eq:TaylBd2},
one could also use the $L^2$-bound with explicit constant 
derived in \eqref{eq:auxEq30}, \eqref{eq:Cexplicit}.

Under hypothesis \eqref{eq:AdAnis} and for $\tau>0$ satisfying \eqref{eq:t>},
$q$-summability of the Taylor coefficients follows.
\[
\begin{aligned}
\sum_{\bsnu\in \IN_0^d} |t_\bsnu|^q
& \leq 
\frac{\|\hat{v}_d(0,\cdot) \|_{L^1(\IR^d)}^q }{(2\pi)^{dq/2}} 
\sum_{\bsnu\in \IN_0^d}
\left( \frac{|\bsnu|!}{\bsnu!}\bsb^\bsnu \right)^{q\rho'}
\\ 
& \leq \frac{\|\hat{v}_d(0,\cdot) \|_{L^1(\IR^d)}^q }{(2\pi)^{dq/2}}  
\sum_{k=0}^\infty \sum_{\bsnu\in \IN_0^d : |\bsnu|=k}
\left( \frac{|\bsnu|!}{\bsnu!}(2 \rho \tau C_1)^{-k/(2 \rho \rho')} \right)^{q \rho'}
\\ 
& \leq \frac{\|\hat{v}_d(0,\cdot) \|_{L^1(\IR^d)}^q }{(2\pi)^{dq/2}}  
\sum_{k=0}^\infty  (2 \rho \tau C_1)^{-q k/(2\rho)} 
\sum_{\bsnu\in \IN_0^d : |\bsnu|=k}
\left( \frac{|\bsnu|!}{\bsnu!} \right)^{q\rho'}.
\end{aligned}
\]	
Using that $|\bsnu|! \geq \bsnu!$ and that $1 \geq q \rho' >0 $ we obtain with the multinomial theorem
\[
\sum_{\bsnu\in \IN_0^d : |\bsnu|=k}
\left( \frac{|\bsnu|!}{\bsnu!} \right)^{q\rho'} \leq \sum_{\bsnu\in \IN_0^d : |\bsnu|=k}
\frac{|\bsnu|!}{\bsnu!} = d^k
\;.
\]
Hence,
provided that 
\begin{equation}\label{eq:tCond} 
\tau > \tau_0(d) \text{ with } \tau_0(d) = \frac{d^{2\rho/q}}{2 \rho C_1},
\end{equation} 
it follows that 
\begin{equation}\label{eq:Tsumq}
\| \{ t_\bsnu \} \|_{\ell^q(\IN_0^d)}^q
=
\sum_{\bsnu\in \IN_0^d} |t_\bsnu|^q
\leq 
\frac{\|\hat{v}_d(0,\cdot) \|_{L^1(\IR^d)}^q}{(2\pi)^{dq/2}} \frac{1}{1-d (2 \rho \tau C_1)^{-q/(2\rho)}}.
\end{equation}
Therefore, we have proved $q$-summability of the Taylor coefficients of the map 
$x\mapsto v_d(\tau,x)$ at $x=0$ for any $\tau > \tau_0(d)$ as in \eqref{eq:tCond}. 
The $q$-norm $ \| \{ t_\bsnu \} \|_{\ell^q(\IN_0^d)}$ is bounded independently of $d$, 
provided that $\tau > \tau_0(d)$ and $\|\hat{v}_d(0,\cdot) \|_{L^1(\IR^d)}(2\pi)^{-d/2}$ 
is bounded independently of $d$.

The $q$-summability \eqref{eq:Tsumq} of the Taylor coefficients of 
$x\mapsto v_d(\tau,x)$ at $x=0$ with $q=1$ 
implies for $\tau > \tau_0(d)$
absolute, pointwise convergence in the cube $[-1,1]^d$ of 
\begin{equation}\label{eq:Tgpc}
v_d(\tau,x)
=
\sum_{\bsnu\in \IN_0^d} t_\bsnu x^\bsnu  \;, \quad x^\bsnu = x_1^{\nu_1}x_2^{\nu_2}....
\end{equation}
Furthermore, 
as was shown in Schwab and Zech \cite[Lemma~2.8]{SZ19_2592},
the fact that the sequence $\{ t_\bsnu \}$ is $q$-summable for some $0 < q < 1$ 
and the coefficient bound \eqref{eq:TaylBd2} imply
that for $\tau > \tau_0(d)$ 
with $\tau_0(d)$ as defined in \eqref{eq:tCond}
exists a sequence $\{ \Lambda_n \}_{n\geq 1} \subset \IN_0^d$ of
\emph{nested, downward closed multi-index sets} (i.e., if $\bse_j\in \Lambda_n$ then $\bse_i\in \Lambda_n$ 
for all $0\leq i\leq j$)
$\Lambda_n \subset \IN_0^d$
with $\#(\Lambda_n) \leq n$ such that general polynomial chaos (gpc for short) approximations
given by the partial sums 
\begin{equation*}
v_d^{\Lambda_n}(\tau,x) = \sum_{\bsnu\in \Lambda_n} t_\bsnu x^\bsnu
\end{equation*}
converge at dimension-independent rate $r=1/q-1$ (see, e.g., Cohen et al.\ \cite[Lemma~5.5]{CDS1})
\begin{equation*}
\sup_{x \in [-1,1]^d}| v_d(\tau,x) - v_d^{\Lambda_n}(\tau,x)|
\leq 
\sum_{\bsnu\in \IN_0^d\backslash \Lambda_n} |t_\bsnu| 
\leq 
n^{-(1/q-1)} \| \{ t_\bsnu \} \|_{\ell^q(\IN_0^d)} \;.
\end{equation*}
The summability \eqref{eq:Tsumq} of the coefficients in the
Taylor gpc expansion \eqref{eq:Tgpc} also implies quantitative bounds on 
the expression rates of ReLU DNNs.
With \cite[Theorem~2.7, (ii)]{SZ19_2592}, 
we find that there exists a constant $C>0$ independent of $d$ 
such that
$$
\sup_{\bsnu\in \Lambda_n} |\bsnu|_1 \leq C\big(1+\log(n)\big) \;.
$$
We now refer to \cite[Theorem~3.9]{SZ19_2592} (with $q$ in place of $p$
in the statement of that result) and, observing that in the
proof of that theorem, only the $p$-summability 
of the Taylor coefficient sequence $\{ t_\bsnu \}$ was used,
we conclude that for $\tau>0$ satisfying \eqref{eq:tCond}
there exists a constant $C>0$ that is independent of $d$
and, for every $n\in \IN$ exists a ReLU DNN $\tilde{v}_d^n$ with 
input dimension $d$, such that 
\begin{equation}\label{eq:ReLUExRt}
\begin{array}{c}
\displaystyle
M(\tilde{v}_d^n) \leq C \Big(1+n\log(n)\log\big(\log(n)\big)\Big),\, \, 
L(\tilde{v}_d^n) \leq C\Big(1+\log(n)\log\big(\log(n)\big)\Big),
\\ \\
\displaystyle
\sup_{x \in [-1,1]^d}| v_d(\tau,x) - \Rl(\tilde{v}_d^n)(x) |
\leq 
C n^{-(1/q-1)} \;.
\end{array}
\end{equation}
%
%%%%%%%%%%%%%%%%%%%%%%%%%%%%%%%%%%%%%55
\section{Conclusion and Generalizations}
\label{seC:ConclGen}
%%%%%%%%%%%%%%%%%%%%%%%%%%%%%%%%%%%%%%%
We proved that prices of European style
derivative contracts on baskets of $d\geq 1$ 
assets in exponential L\'evy models
can be expressed by ReLU DNNs to accuracy $\varepsilon > 0$
with DNN size polynomially growing in $\varepsilon^{-1}$ and $d$,
thereby overcoming the curse of dimensionality. 
The technique of proof was based on  probabilistic arguments and provides 
expression rate bounds that scale algebraically
in terms of the DNN size.
We then also provided an alternative, analytic argument, 
that allows to prove \emph{exponential expressivity of ReLU DNNs} 
of the option price, i.e.\ of the map $s\mapsto u(t,s)$ 
at any fixed time $0<t<T$, with DNN size growing
polynomially w.r.\ to $\log(\varepsilon)$ to achieve
accuracy $\varepsilon > 0$.
For sufficiently large $t>0$, based on analytic arguments 
involving parabolic smoothing and sparsity of generalized polynomial
chaos expansions, we established in \eqref{eq:ReLUExRt} a second,
algebraic expression rate bound for ReLU DNNs that is free from
the curse of dimensionality. 
In a forthcoming work Gonon and Schwab \cite{GononChS2} we 
address PIDEs \eqref{eq:MultdPDEx} with non-constant coefficients. 
In addition, the main result of the present paper,  
Thm.\ref{prop:Ddresult}, could be extended in the following directions.

First, the expression rates are, almost certainly, not optimal in general;
for high-dimensional diffusions, which are a
particular case with $A^d = I$ and $\nu^d = 0$,
in Elbr\"achter et al.\ \cite{EGJS18_787} we established for particular payoff functions
a spectral expression rate in terms of the DNN size,
free from the curse of dimensionality.

Solving Hamilton-Jacobi partial integrodifferential equations (HJPIDEs for short) by DNNs: 
it is classical that the Kolmogorov equation for the 
exponential LP $X^d$ in Section \ref{sec:ExpLP} is, in fact, a
special case of a HJPIDE (e.g.\ Barles et al.\ \cite{BBP1997}, Barles and Imbert \cite{BarImb2009}). 
In a forthcoming work \cite{GononChS2} we aim at proving that the
expression rate bounds obtained in Section \ref{sec:DNNMultivLevy}
imply corresponding expression rate bounds for ReLU DNNs 
which are free from the curse of dimensionality 
for viscosity solutions of general HJPIDEs associated to the LP $X^d$ and 
for its exponential counterparts. 

Barriers: 
We considered payoff functions corresponding to European style contracts. 
Here, the stationarity of the LP $X^d$
and exponential L\'evy modelling allowed to reduce our analysis
to Cauchy problems of the Kolmogorov equations of $X^d$ in 
$\IR^d$. 
In L\'evy models in the presence of barriers,
option prices generally exhibit singularities at the barriers.
More involved versions of the Fourier transform based representations 
are available (involving a so-called Wiener-Hopf factorization of the 
Fourier symbol, see, e.g., Boyarchenko and Levendorski\u{\i} \cite{BoyLevBarrLevy2002}).
For LPs $X^d$ with bounded exponential moments,
the present regularity analysis may be localized to compact subsets, 
well separated from the barriers, subject to an exponentially small 
localization error term; see Hilber et al.\ \cite[Chapter~10.5]{CMQFBook}. 
Here, the semiheavy tails of the LPs $X^d$ enter crucially in the analysis.
We therefore expect the present 
DNN expression rate bounds to remain valid also for barrier contracts,
at least far from the barriers, for the LPs $X^d$ considered here.

Dividends:
We assumed throughout that contracts do not pay dividends; however,
including a dividend stream (with constant over $(0,T]$ rate) 
on the underlying does not change the mathematical arguments;
we refer to Lamberton and Mikou \cite[Section~3.1]{LambMik08} for a complete
statement of exponential L\'{e}vy models with constant dividend payment rate $\delta > 0$, 
and for the corresponding pricing of European and American style contracts 
for such models.

American style contracts:
Deep learning based algorithms for the numerical solution
of optimal stopping problems for Markovian models
have been recently proposed in Becker et al.\ \cite{Becker2019}.
For the particular case of American style contracts
in exponential L\'evy models, \cite{LambMik08} provide an analysis in the
univariate case, and establish qualitative properties of the exercise boundary
$\{ (b(t),t): 0<t<T \}$.

Here, for geometric L\'evy models, in certain situations 
($d=1$, i.e.\ single risky asset,
monotonic, piecewise analytic payoff function) the option price,
as a function of $x\in \IR$ at fixed $0<t<T$,
is shown in \cite{LambMik08}
to be a piecewise analytic function which is, globally, H\"older continuous
with a possibly algebraic singularity at the exercise boundary $b(t)$.
This holds, likewise, for the price expressed 
in the logarithmic coordinate $x=\log(s)$.
The ReLU DNN expression rate
of such functions has been analyzed in Opschoor et al.\ \cite[Section~5.4]{Opschoor2020}.
In higher dimensions $d>1$, recently also higher H\"older regularity
of the price in symmetric, stable 
L\'evy models has been obtained for smooth payoffs in Barrios et al.\ \cite{RosOto2018}.

{\small 
\bibliographystyle{amsalpha}
\bibliography{references}
}
\end{document}